\documentclass[11pt]{amsart}

\usepackage{amsmath,amscd,amssymb}
\usepackage[mathscr]{eucal}
\usepackage[all]{xy}
\usepackage{enumerate}

\newcommand{\vanish}[1]{\relax}

\newcommand{\N}{\mathbb{N}}

\newcommand{\R}{\mathbb{R}}
\newcommand{\C}{\mathbb{C}}
\newcommand{\E}{\mathbb{E}}

\newcommand{\Sum}[2][\relax]{%
 \ifx#1\relax \sideset{}{_{#2}}\sum
 \else \sideset{}{^{#1}_{#2}}\sum
 \fi}

\DeclareMathOperator{\Sect}{Sect}

\DeclareMathOperator{\dom}{dom}
\DeclareMathOperator{\ran}{ran}

\newcommand{\norm}[2][\relax]{%
   \ifx#1\relax \ensuremath{\left\Vert#2\right\Vert}
   \else \ensuremath{\left\Vert#2\right\Vert_{#1}}
   \fi}

\makeatletter
\newcommand{\sprod}[2]{\ensuremath{%
  \setbox0=\hbox{\ensuremath{#2}}
  \dimen@\ht0
  \advance\dimen@ by \dp0
  \left(\left.#1\rule[-\dp0]{0pt}{\dimen@}\,\right|#2\hspace{1pt}\right)}}
 \makeatother

\newcounter{aufzi}

\newcounter{aufzii}

\newcounter{aufziii}

 \newtheorem{thm}{Theorem}[section]
 \newtheorem{cor}[thm]{Corollary}
 \newtheorem{lemma}[thm]{Lemma}
 \newtheorem{prop}[thm]{Proposition}

 \theoremstyle{definition}
 \newtheorem{defn}[thm]{Definition}

 \theoremstyle{remark}

\newtheorem{example}[thm]{Example}

\newtheorem{remark}[thm]{Remark}

\numberwithin{equation}{section}

\numberwithin{equation}{section} \numberwithin{theorem}{section}

\newcommand\NPp{\mathcal{NP}_+}

\begin{document}

\title[Resolvents of functions of sectorial operators
]
{
Resolvent representations for functions of sectorial operators}

\author{Charles Batty}
\address{St. John's College\\
University of Oxford\\
Oxford OX1 3JP, UK
}

\email{charles.batty@sjc.ox.ac.uk}

\author{Alexander Gomilko}
\address{Faculty of Mathematics and Computer Science\\
Nicolas Copernicus University\\
ul. Chopina 12/18\\
87-100 Toru\'n, Poland \\
and Institute of Telecommunications and Global \\
Information Space, National Academy of Sciences of Ukraine\\
Kiev, Ukraine
}
\email{gomilko@mat.umk.pl}



\author{Yuri Tomilov}
\address{\textbf{}
Institute of Mathematics\\
Polish Academy of Sciences\\
\'Sniadeckich 8\\
00-956 Warszawa, Poland
}

\email{ytomilov@impan.pl}

\thanks{This work was completed with the support of the  EPSRC grant EP/J010723/1,  the NCN grant
DEC-2014/13/B/ST1/03153 and the EU grant ``AOS'', FP7-PEOPLE-2012-IRSES, No 318910.  The authors are grateful to two anonymous referees for helpful remarks and suggestions.}



\date{\today}

\begin{abstract}

We obtain integral representations for the resolvent of $\psi(A)$,
where $\psi$ is a holomorphic function mapping the right half-plane and the right half-axis into themselves,
and $A$ is a sectorial operator on a Banach space.
As a corollary, for a wide class of functions $\psi$,
we show that the operator $-\psi(A)$ generates a sectorially bounded holomorphic $C_0$-semigroup
on a Banach space whenever $-A$ does, and the sectorial angle of $A$ is preserved.  When $\psi$ is a Bernstein function, this was recently proved by Gomilko and Tomilov, but the proof here is more direct.
Moreover, we prove that such a permanence property for $A$ can be described, at least on Hilbert spaces,
in terms of the existence of a bounded $H^{\infty}$-calculus for $A$. As byproducts of our approach, we also
obtain new results on functions mapping generators of bounded semigroups into generators of holomorphic semigroups and on subordination for Ritt operators.
\end{abstract}

\maketitle

\section{Introduction}
Given a sectorial operator $A$ of angle $\theta \in [0,\pi)$ on a Banach space $X$, one usually defines a holomorphic functional calculus for $A$  as a mapping
\[
\mathcal A \ni f \mapsto f(A),
\]
where $\mathcal A$ is an appropriate algebra  of holomorphic functions on a sector $\Sigma_\psi:=\{z \in \mathbb C: |\arg (z)| < \psi\}$, $\psi \in (\theta,\pi]$, and $f(A)$ is, in general, a closed operator on $X$. Usually certain regularisations on $f$ are necessary and the definition of $f(A)$ may be quite implicit.

For a formal procedure as above to be useful one often needs further conditions on $f$ ensuring that
 $f(A)$ is bounded and then leading to sharp enough estimates for $f(A)$.  In this way, we come naturally to the established notions of holomorphic functional calculus of sectorial operators, $H^\infty$-calculus and its boundedness, Fourier multipliers, and other aspects of the theory of functional calculi during the last twenty years.

It is natural to wonder what other properties of $f(A)$ one might expect if $A$ is fixed.
In particular, to deal with $f(A)$ within the established framework of functional calculus, one needs to know that $f(A)$ is at least sectorial, and desirably  of the same angle as $A$.   This direction of research has  received little attention.

The aim of the present paper is to describe classes $\mathcal F$ of holomorphic functions preserving the set $\operatorname{Sect}(\theta)$ of sectorial operators of angle $\theta$  within the framework of the holomorphic functional calculus. Motivated by applications to semigroup theory, the following four natural questions  arise in such a study.
\begin{itemize}
\item [(Q1)] Which functions preserve $\operatorname{Sect}(\theta)$ for $\theta \in [0,\pi/2)$? In other words, when do  they preserve the class of (negative) generators of sectorially bounded holomorphic $C_0$-semigroups of angle $\pi/2-\theta$?

\item [(Q2)] Which functions preserve  $\operatorname{Sect}(\theta)$ for  $\theta \in [0,\pi)$?

\item [(Q3)] Which functions preserve the class of (negative) generators of bounded semigroups?

\item [(Q4)] Which functions have a so-called improving property, meaning that they map $\operatorname{Sect}(\theta)$ for some $\theta \in [\pi/2,\pi)$ into (negative) generators of sectorially bounded holomorphic $C_0$-semigroups?
    \end{itemize}

Simple considerations with  $X = \C$ show that a holomorphic function possessing any of the permanence properties in (Q1)-(Q4) has to map the right half-plane $\C_+ := \{\lambda\in\C : \operatorname{Re}\lambda > 0\}$ into itself.  Moreover to preserve the angle of sectoriality it has to map the positive half-axis into itself.  The class of functions with these properties, called $\mathcal{ NP}_+$ in the paper, and its subclasses, such as the Bernstein functions and their relatives,  will be basic for our treatment.  The $\mathcal{NP}_+$-functions allow a more instructive description as the functions which are holomorphic in $\mathbb C_+$ and which preserve the sector $\Sigma_\theta$ for each $\theta \in [0,\pi/2)$.  Thus they are very natural candidates for answering any of the questions above.
The $\mathcal{NP}_+$-functions have been treated in the literature under several different names, see Section \ref{pnp} for a clarification and comments on that.  They were used recently in \cite{GT2} for the purposes of holomorphic functional calculus, and here we develop that theme.

Although the relevance of ${\mathcal NP}_+$-functions has only recently become apparent, some subclasses of ${\mathcal NP}_+$-functions have already played important roles in the study of permanence properties.  Bernstein functions can be defined to be those functions $f$ such that the functions $e^{-tf}$ for $t>0$ are the Laplace transforms of a convolution semigroup of sub-probability measures $\mu_t$ on $[0,\infty)$, and so they are clearly in ${\mathcal NP}_+$.  Bochner noted in \cite{B} that this definition can serve as a base for constructing operator semigroups generated by the negatives of operator Bernstein functions thus introducing the notion of subordination.   In \cite{Phil}, extending Bochner's ideas and dealing with (Q3), Phillips proved that Bernstein functions preserve the class of negative generators of bounded $C_0$-semigroups; see also \cite{Ro} and \cite{Ne} for related studies.  In \cite{GT} subordination was included in the framework of functional calculi in the sense of \cite{Ha06}  or \cite{KW04}.
 
Bernstein functions do not extend holomorphically outside $\mathbb C_+$ in general.  However, there is an important subclass of them, called complete Bernstein functions, which extend holomorphically to $\mathbb C \setminus (-\infty,0]$ and, moreover, map the upper half-plane into itself.   By restricting to the class of complete Bernstein functions in \cite{H}, Hirsch proved a permanence property of the type considered in (Q2), but without addressing the preservation of the corresponding angles.  In \cite{KR} the authors asked whether Bernstein functions preserve the class of negative generators of holomorphic semigroups without being specific about the interpretation of their question.  In \cite{BBL} Berg, Boyadzhiev, and de Laubenfels strengthened  part of Hirsch's result and showed that complete Bernstein functions preserve $\operatorname{Sect}(\theta)$ if $\theta < \pi/2$, and they obtained partial results for some other Bernstein functions.  This provided partial answers to (Q1) and to one interpretation of the question from \cite{KR}.

As far as Bernstein functions are concerned, conclusive answers to (Q1) and (Q2) were obtained very recently in \cite{GT}. It was proved there that complete Bernstein functions preserve $\operatorname{Sect}(\theta)$ for any fixed $\theta \in [0,\pi)$, and moreover Bernstein functions do the same if the angle  $\theta$ is less than $\pi/2$. 

Improving properties of $\mathcal{NP}_+$-functions as described in (Q4) were studied by a number of authors including Yosida \cite{Yo60}, Paquet \cite{Pa}, Carasso and Kato \cite{CK}, Fujita \cite{F} and Mirotin \cite{M1}, \cite{Mir2}.   A number of geometric and explicit conditions, sometimes characterizing the improving property, were given in \cite{GT}. See \cite{GT} for discussion of those papers and more references.

The study of permanence of functional calculi in \cite{GT} was based on the specific structure of Bernstein functions and their properties within several compatible functional calculi including the extended Hille-Phillips calculus as described in \cite[Section 3.3]{GT}.  The function $1/z$ is not a Bernstein function, but it is an $\NPp$-function.  If $A$ is a sectorial operator with dense range, then $A^{-1}$ is sectorial of the same angle.  Moreover there are very natural examples of ${\mathcal NP}_+$-functions that are neither Bernstein nor the reciprocal of a Bernstein function, and map sectorial operators of a fixed angle into themselves.  For example $f(z)=\sqrt {z (1-e^{-z})}$ has these properties and it is the geometric mean of two Bernstein functions (see Example 4.8 and Corollary \ref{TCbf13} below).  Thus the extensions of (Q1)-(Q4)  to the class  of $\NPp$-functions are very natural, and we address them in this paper.  Since there are examples of $\NPp$-functions which lie outside the scope of the extended Hille-Phillips calculus, we work in the setting of the holomorphic functional calculus for $\mathcal{NP}_+$-functions and we use different techniques from \cite{GT}.

Nevertheless one cannot expect positive answers to (Q1)-(Q4) for all functions in ${\mathcal NP}_+$. The underlying reason is that such properties are essentially equivalent to the boundedness of $H^\infty$-calculus as shown by the following result which is proved in Theorem \ref{ops}.

\begin{thm}
Let $A$ be a sectorial operator of angle $\omega \in [0, \pi/2)$ on a Banach space $X$ with dense range,
and let $\theta \in (\omega, \pi/2)$. Consider the following statements.
\begin{itemize}
 \item [(i)] $A$ admits a bounded $H^\infty$-functional calculus on $\Sigma_\theta$.
\item [(ii)] For every $f \in \mathcal{NP}_+$,  $f(A)$ is a sectorial operator of angle (at most) $\omega$.
\item [(iii)] For every $f \in  \mathcal{NP}_+$, $-f(A)$ is the generator of a bounded $C_0$-semigroup.
\item [(iv)] For every  $f\in \mathcal{NP}_+$ such that $\lim_{t \downarrow 0} f(t)$ and $\lim_{t \to \infty} f(t)$ both exist in $(0, \infty)$, $-f(A)$ is the generator of a $C_0$-semigroup.
\item [(v)] $A$ has  bounded $H^\infty$-calculus on $\mathbb C_+$.
\end{itemize}
Then
$$
{\rm(i)}  \implies {\rm(ii)} \implies {\rm(iii)} \iff {\rm(iv)} \iff {\rm(v)}.
$$
If $X$ is a Hilbert space, all four properties are equivalent.
\end{thm}

 Since there are well-known examples of sectorial operators $A_0$ with angle $0$ without a bounded $H^\infty$-calculus, the operator $f(A_0)$ is not sectorial for certain  $\mathcal{NP}_+$-functions $f$. See Theorem \ref{ops} and Remark \ref{opsrem} for further details. Thus, as far as the permanence properties are concerned, one has to look for admissible \emph{subclasses} of $\mathcal{ NP}_+$-functions,
and the theorem above may serve as an illustration of the difficulty of the problems that we deal with.

Given a $\mathcal{NP}_+$-function $f$, we start by obtaining an integral representation for the resolvent of $f(z)$ (Lemma \ref{In2})  and then we try to replace $z$ by $A$. The representation is of independent interest, and it may be useful in other contexts as well.  If the representation converges absolutely with suitable estimates, we are then able to extend it holomorphically to an appropriate sector, and to obtain the corresponding resolvent estimate for $f(A)$ rather directly. However, to ensure the absolute convergence and to transfer the sectoriality estimates from $A$ to $f(A)$ we work with a proper subclass $\mathcal D$ of  $\mathcal { NP}_+$ that still have some traces of the special behaviour of Bernstein functions. The class $\mathcal D$ is defined and studied  thoroughly in Section  \ref{pnp}.

As we show in Section \ref{BF}, the class $\mathcal D$ contains not only the Bernstein functions and their reciprocals, but also several function classes of interest beyond Bernstein functions. At the same time, the structure of $\mathcal D$ allows us   to get a very explicit resolvent estimate for $f(A)$, leading eventually to sectoriality of $f(A)$.   Thus we substantially extend one of the main results from \cite{GT}, Theorem $1.1$, and provide an alternative (and more direct) proof 
 based on function-theoretic arguments. On the other hand, it is sometimes of interest to drop the assumption of sectorial boundedness and to consider holomorphic semigroups that are merely bounded on the real half-axis. While the latter class was  handled successfully in \cite{GT}, it is out of reach for the approach of this paper (see the remark following Corollary \ref{TCbf13}).
On the other hand, our resolvent approach allows us to prove a new result on improving properties of Bernstein functions extending Theorem $1.3$ from \cite{GT}.

The following statement gives a flavour of our main results proved in Section \ref{pres} (see Corollaries \ref{TCbf13} and \ref{improve}). 
\begin{thm}  Assume that $F_j,  1 \le j \le n$, are Bernstein functions and
define
\[
F(z)=F_1(z^{\alpha_1})^{\beta_1} \dots F_n(z^{\alpha_n})^{\beta_n},  \qquad z \in \C_+,
\]
where
$0 < \alpha_j \le 1$, $0 < \beta_j \le 1$, and $\sum_{j=1}^{n} \alpha_j\beta_j \le 1$.  
Let $A$ be a sectorial operator of angle $\omega$, and assume that $A$ has dense range.
\begin{enumerate}[\rm1.]
\item  If $\omega \in [0,\pi/2)$ so that $-A$ generates a sectorially bounded holomorphic semigroup of angle $\pi/2-\omega$,
then $-F(A)$ also generates a sectorially bounded holomorphic semigroup of angle $\pi/2-\omega$.
\item  If  $\alpha \in (1/2, 1)$,  $\omega \in (0,\pi/(2\alpha))$ and 
\[
G(z) = F(z^\alpha), \qquad z \in \Sigma_{\pi/(2\alpha)},
\]
 then $-G(A)$ is the generator of a sectorially bounded holomorphic
$C_0$-semigroup of angle $\pi/2 - \alpha \omega$.
\end{enumerate}
\end{thm}

For more statements as above and their discussion, see Section \ref{pres} below.
Note that our approach also allows us to obtain a continuous counterpart of one of the main results in \cite{GT2} saying that ``barycentres'' of sectorially bounded holomorphic semigroups are Ritt operators (Theorem \ref{TCbf14}).

\section{Positive Nevanlinna--Pick functions} \label{pnp}



In the literature, a holomorphic function $g$ mapping the upper half-plane $\C^+ := \{\lambda\in\C : \operatorname{Im}\lambda>0\}$ to $\C^+ \cup \R$ is called a \emph{Nevanlinna function}, a \emph{Pick function}, a \emph{Herglotz function}, an \emph{R-function} or various combinations of these names.  Such a function is either a real constant, or it maps $\C^+$to $\C^+$.    We can easily transfer properties of the functions $g$ on $\C^+$ by putting $f(z) = -ig(iz)$.  We shall say that a holomorphic function $f: \C_+ \to \C_+$ mapping $(0,\infty)$ to $(0,\infty)$ is a {\it positive Nevanlinna--Pick function}, or an $\NPp$-function, and we will write $f \in \NPp$. We can ignore the degenerate case of the constant function $0$, so we consider only $\NPp$-functions as defined here.

In the literature on electrical networks, holomorphic functions from $\C_+$ to $\C_+$ are called \emph{positive} functions, and positive Nevanlinna--Pick functions are called \emph{positive real} functions.  However the positivity involved here is the very weak notion of preserving the sign of the real part of the variable.      We do not follow this terminology because it obscures the crucial property for us that our functions are holomorphic functions of a complex variable, and it is natural that they satisfy the strong notion of positivity by mapping positive real numbers to positive real numbers.  We hope that this slight abuse of terminology when working on $\C_+$ will not cause confusion.

It is clear that the class $\NPp$  is closed under sums, positive scalar multiples, reciprocals, and compositions, and that any function $f \in \NPp$ satisfies
\[
f(\overline z) = \overline{f(z)}, \qquad z \in \C_+.
\]
We shall use these properties without further comment.    Other properties of $\NPp$-functions are described in \cite{Br07} and \cite[Section 3.1]{GT2}, and in Proposition \ref{pnpp} below.   Note in particular that $f \in \NPp$ if and only if $f$ has a representation
\begin{equation} \label{npp1}
f(z) = \int_{-1}^1 \frac{2 z}{(1+z^2) + t(1-z^2)} \, \rho(dt), \qquad z \in \C_+,
\end{equation}
for some (unique) finite positive measure $\rho$ on $[-1,1]$ \cite[Theorem 3.1]{Br07}.  Here and throughout the paper, measures on subsets of $\R$ are Borel measures.  Under a change of variable this representation is converted to 
\begin{equation} \label{npp2}
f(z) = az + \frac{b}{z} + 2z \int_{0+}^\infty \frac{1+s^2}{s^2+z^2} \, \mu(ds),  \qquad z \in \C_+,
\end{equation} 
where $\mu$ is a finite positive measure on $(0,\infty)$ (see \cite{Ca32}).   Here and subsequently, we use the notation $\int_{0+}^{\infty}$ to denote an integral over $(0,\infty)$, following \cite{SchilSonVon2010}.   Each of these representations provides an affine isomorphism between the set of $\NPp$-functions $f$ with $f(1)=1$ and the set of probability measures on either $[-1,1]$ or $[0,\infty]$.  Since the extreme points of the latter set are the unit point masses, it follows that the extremal $\NPp$-functions with $f(1) = 1$ are of the form ${2z}((1+z^2) + t(1-z^2))^{-1}$ for $t \in [-1,1]$; or equivalently of the form $z(1+s^2)(z^2+s^2)^{-1}$ for $s \in [0,\infty)$ together with the identity function $z$.  Except for $t=\pm1$ (or $s=0$ and the case $z$), these extremal functions are not monotonic, but they are obtained from the identity function $z$ by simple operations (reciprocals, sums, and scalar multiples).  

The representation (\ref{npp2}) is related to Cauchy transforms.  Let $\nu$ be a positive measure on $\R$ with 
\[
\int_{-\infty}^\infty \frac{\nu(ds)}{1+|s|} < \infty,
\]
and define
\[
f(z) = \frac{1}{2\pi} \int_{-\infty}^\infty \frac{\nu(ds)}{z-is}, \qquad z \in \C_+.
\]
Then $f$ is a holomorphic function from $\C_+$ to $\C_+$, and it belongs to $\NPp$ if and only if $\nu$ is symmetric under the mapping $s \mapsto -s$.   If $\nu$ is symmetric and $\nu(ds) = 2\pi(1+s^2)\,\mu(ds)$ on $(0,\infty)$, $a=0$ and $b = \nu(\{0\})$, then we obtain (\ref{npp2}).

For an $\NPp$-function $f$, we shall let
\[
f(0+) := \lim_{t\to0+} f(t),  \qquad  f(\infty) := \lim_{t\to\infty} f(t),
\]
when the limits exist in $[0,\infty]$.  However the limits may not exist. Let $\mu$ be a finite positive measure on $(0,\infty)$. Consider the $\NPp$-function
\[
f(z) = \frac{2z}{\pi} \int_{0+}^\infty \frac{1+s^2}{s^2+z^2} \, \mu(ds),  \qquad z\in\C_+.
\]
When $\mu$ is extended symmetrically to $\R$ and $z = t+iy$, the real part of this function is the Poisson integral
$$
u(t,y) = \frac{1}{\pi} \int_{-\infty}^\infty \frac{t(1+s^2)}{t^2+(y+s)^2} \, \mu(ds), \quad t>0, \; y\in\R.
$$
Since $u(t) = f(t)$ for $t>0$ and $\mu$ is symmetric on $\R$, a theorem of Fatou and its converse (see \cite[Chapter IV]{Do74}, \cite[p.257]{Rud}, \cite{Lo43}, \cite{Do63}), show that the following are equivalent for $L \in [0,\infty)$:
\begin{enumerate} [\rm(i)]
\item  $\lim_{t\to0+} f(t) = L$;
\item  The non-tangential limits $\lim_{z\to0, z \in \Sigma_\theta} f(z) = L$ for $\theta \in (0, \pi/2)$;
\item  The derivative  $\lim_{a\to0+} \mu(0,a)/a = L$.
\end{enumerate}
For $L=\infty$, (iii) implies (i), but the converse does not hold \cite[Section III]{Rud}.
There are measures $\mu$ which fail to satisfy (iii).  For example, one may take the measure $\mu = \sum_{n=0}^\infty 2^{-n} \delta_{2^{-n}}$, which gives an $\NPp$-function $f$ which is bounded on $(0,\infty)$ but the limit in (iii) does not exist and accordingly $f(0+)$ does not exist.  The function $g(z)=f(1/z)$ belongs to $\NPp$ and $g(\infty)$ does not exist.

We will need some specific properties of $\NPp$-functions as follows.  Here and throughout the paper we use the notation
\[
\Sigma_\theta := \begin{cases}  \{z \in \C: |\arg(z)| < \theta\}, & \theta \in (0,\pi], \\ (0,\infty), &\theta=0. \end{cases}
\]
Note that $\Sigma_{\pi/2} = \C_+$.

\begin{prop}  \label{pnpp}
Let $f \in \NPp$.
\begin{enumerate}  [\rm1.]
\item \label{pnpp1} $f$ maps $\Sigma_\theta$ into itself for all $\theta \in (0,\pi/2)$.
\item \label{pnpp2}  If $\alpha \in (0,1)$ and $g_\alpha(z) = f(z^\alpha)^{1/\alpha}$ for $z \in \C_+$, then $g \in \NPp$.  
\item \label{pnpp3} Let $r, t>0$, $\theta \in (-\pi/2,\pi/2)$, and $\rho :=\min \left( \dfrac{r}{t},\dfrac{t}{r} \right) $. 
\begin{enumerate} [\rm(a)]
\item \label{pnpp3a} If $\cos 2\theta \ge \rho^2$, then
\[
\rho f(t)   \le  |f(re^{i\theta})| \le \frac{f(t)}{\rho}.
\]
In particular when $\theta=0$,
\begin{equation} \label{ts3}
\rho f(t) \le f(r) \le \frac{f(t)}{\rho} .
\end{equation}
\item \label{pnpp3b} If $\cos 2\theta \le \rho^2$, then
\[
 f(t)rt \left| \frac{\sin 2\theta}{t^2-r^2e^{2i\theta}} \right| \le |f(re^{i\theta})| \le \frac {f(t)}{rt} \left|\frac{t^2-r^2 e^{2i\theta}}{\sin 2 \theta} \right|.
\]
In particular when $r=t$,
\begin{equation} \label{ts4}
f(t) \cos\theta \le |f(te^{i\theta})| \le f(t)/\cos\theta.
\end{equation}
\end{enumerate} 
\item \label{pnpp4} Let $\theta \in (0,\pi/2)$.
\begin{enumerate} [\rm(a)]  
\item \label{pnpp4a} Assume that $f(0+)$ exists in $[0,\infty]$.  Then
\[
\lim_{z\to0,z\in \Sigma_\theta} f(z) = f(0+).
\]
\item \label{pnpp4b} Assume that $f(\infty)$  exists in $[0,\infty]$.  Then
\[
\lim_{|z|\to\infty,z\in \Sigma_\theta} f(z) = f(\infty).
\]
\end{enumerate}
\end{enumerate}
\end{prop}

\begin{proof}  (\ref{pnpp1}) was proved in \cite[Theorem VI, p.37]{Br31}.  It can alternatively be deduced from the three-lines theorem. (\ref{pnpp2}) follows easily from (\ref{pnpp1}).  (\ref{pnpp3}) was proved in \cite[Theorem 3.2]{Br07} (for $t=1$, and the general case follows easily), improving earlier results by providing sharp bounds.  The upper bounds are obtained by establishing them for the extremal functions and applying (\ref{npp2}), and the lower bounds are deduced by applying the upper bounds to $1/f$.

When $f(0+) \in \{0,\infty\}$, (\ref{pnpp4a}) follows from (\ref{pnpp3b}).  The general case of (\ref{pnpp4a}) is a classical fact \cite[Theorem IV, p.49]{Do74}, and (\ref{pnpp4b}) follows from (\ref{pnpp4a}) by considering $f(1/z)$.  
\end{proof}

The following inequality for complex numbers is elementary.  If $\theta\ge0$, $\psi\ge0$ and $\theta+\psi<\pi$, then
\begin{equation}\label{ts1}
|z+\lambda|\ge \cos\big((\theta+\psi)/2\big)\,(|z|+|\lambda|),\quad
z\in \Sigma_\theta,\quad \lambda\in \Sigma_\psi.
\end{equation}
Let $f \in \NPp$. For $\theta\in [0,\pi)$, $\psi\in [0,\pi/2)$ such that
$\theta+\psi<\pi$, $z\in \Sigma_\theta$ and  $\lambda\in \Sigma_\psi$, Proposition \ref{pnpp}(\ref{pnpp1}), (\ref{ts1}) and (\ref{ts4}) give
\begin{align}\label{Pest}
|z+f(\lambda)|&\ge \cos\big((\theta+\psi)/2\big)\,(|z|+|f(\lambda)|) \\
&\ge  \cos\psi \, \cos\big((\theta+\psi)/2\big) \, (|z|+f(|\lambda|)) \nonumber.
\end{align}

The following lemma gives a representation of the resolvent of an $\NPp$-function in a form which will be useful when we consider sectorial operators $A$ instead of scalars $\lambda$.   We shall see in Proposition \ref{limits} and Corollary \ref{limit0} that the assumptions of existence of $f(\infty)$ and $f(0+)$ are natural for the study of sectorial operators.

\begin{lemma}\label{In2}
Let $f\in \NPp$, $q>2$, $\lambda\in \Sigma_{\pi/q}$, and $z \in \Sigma_{\pi- \pi/q}$.  
\begin{enumerate}[\rm1.]
\item  Assume that $f(\infty)$ exists.
Then 
\begin{multline}\label{R11+}
(z+f(\lambda))^{-1} \\
=\frac{1}{z+f(\infty)}
+\frac{q}{\pi}\int_0^\infty
\frac{\operatorname{Im}f(te^{i\pi/q})\,t^{q-1}}
{(z+f(te^{i\pi/q}))(z+f(te^{-i\pi/q}))(\lambda^q+t^q)} \,dt,
\end{multline}
where the integral may be improper.
\item  Assume that $f(0+)$ exists.
Then 
\begin{multline} \label{R11++}
(z+f(\lambda))^{-1} \\
=\frac{1}{z+f(0+)}
- \frac{q}{\pi}\int_0^\infty
\frac{\operatorname{Im}f(te^{i\pi/q})\,\lambda^q}
{(z+f(te^{i\pi/q}))(z+f(te^{-i\pi/q}))(\lambda^q+t^q)t} \,dt,
\end{multline}
where the integral may be improper.
\end{enumerate}
\end{lemma}

\begin{proof}
We first consider the case when $f(\infty)=\infty$.  Take $\theta \in (q|\arg\lambda|,\pi)$, so $\lambda^q \in \Sigma_\theta$.  By Proposition \ref{pnpp}(\ref{pnpp1}),
\[
f(\zeta^{1/q})\in \overline{\Sigma}_{\theta/q},\quad
\zeta\in \overline{\Sigma}_\theta \setminus \{0\},
\]
so $(z+f(\zeta^{1/q}))^{-1}$ is holomorphic and bounded for $\zeta \in \Sigma_\theta$.  For  $R>|\lambda|^q$, let $\gamma_R$ be the boundary of $\{\zeta \in \Sigma_\theta : |\zeta|<R$\}.  By Cauchy's theorem,
\begin{equation}\label{Ca+}
(z+f(\lambda))^{-1}=
\frac{1}{2\pi i}
\int_{\gamma_R}\,
\frac{d\zeta}{(z+f(\zeta^{1/q}))(\zeta-\lambda^q)}.
\end{equation}
Deforming the contour $\gamma_R$ to the negative semi-axis, we obtain that
\begin{align} \label{RR+}
(z+f(\lambda))^{-1} &=\frac{1}{2\pi i}
\int_{|\zeta|=R}\,\frac{d\zeta}{(z+f(\zeta^{1/q}))(\zeta-\lambda^q)}
\\
&\phantom{XXX} -\frac{1}{2\pi i}\int_0^R\,
\frac{ds}{(z+f(s^{1/q}e^{i\pi/q}))(s+\lambda^q)}  \nonumber
\\
&\phantom{XXX} +\frac{1}{2\pi i}
\int_0^R\,\frac{ds}{(z+f(s^{1/q}e^{-i\pi/q}))(s+\lambda^q)}  \nonumber
\\
&=\frac{1}{2\pi i}\int_{|\zeta|=R}\,
\frac{d\zeta}{(z+f(\zeta^{1/q}))(\zeta-\lambda^q)}  \nonumber
\\
&\phantom{XXX} +\frac{1}{\pi}\int_0^R\,
\frac{\operatorname{Im}f(s^{1/q}e^{i\pi/q})}
{(z+f(s^{1/q}e^{i\pi/q}))(z+f(s^{1/q}e^{-i\pi/q}))
(s+\lambda^q)} \, ds. \nonumber
\end{align}
Since $f(\infty)=\infty$ and $q>2$, Proposition \ref{pnpp}(\ref{pnpp4b}) gives 
\[
\lim_{|\zeta|\to\infty} \,\frac{1}{f(\zeta^{1/q})}=0
\quad \mbox{uniformly for}\quad \zeta\in \Sigma_\pi.
\]
Hence
\begin{equation}\label{infty+}
\lim_{R\to\infty}\,\int_{|\zeta|=R}\,
\frac{d\zeta}{(z+f(\zeta^{1/q}))(\zeta-\lambda^q)}=0.
\end{equation}

Letting $R\to\infty$ in (\ref{RR+}), we obtain as improper integrals:
\begin{align*}
(z+f(\lambda))^{-1}&=\frac{1}{\pi}\int_0^\infty\,
\frac{\operatorname{Im}f(s^{1/q}e^{i\pi/q})}
{(z+f(s^{1/q}e^{i\pi/q}))(z+f(s^{1/q}e^{-i\pi/q}))
(s+\lambda^q)} \,ds
\\
&=\frac{q}{\pi}\int_0^\infty\,
\frac{t^{q-1}\,\operatorname{Im}f( te^{i\pi/q})}
{(z+f(te^{i\pi/q}))(z+f(t e^{-i\pi/q}))
(t^q+\lambda^q)} \,dt,
\end{align*}
i.e.\ (\ref{R11+}) holds.
 
When $f(\infty)<\infty$, \eqref{R11+} can be proved in a similar manner.
Instead of (\ref{Ca+}), we write
\begin{multline*}
(z+f(\lambda))^{-1}-(z+f(\infty))^{-1}
\\
=\frac{1}{2\pi i}
\int_{\gamma_R}\,
\left(
\frac{1}{z+f(\zeta^{1/q})}-\frac{1}{z+f(\infty)}\right)\,
\frac{d\zeta}{\zeta-\lambda^q},
\end{multline*}
and, instead of (\ref{infty+}), we  use Proposition \ref{pnpp}(\ref{pnpp4b}) to obtain that
\begin{align*}
\lefteqn{\lim_{R\to\infty}\,\int_{|\zeta|=R}\, \left(
\frac{1}{z+f(\zeta^{1/q})}-\frac{1}{z+f(\infty)}\right)\,
\frac{d\zeta}{\zeta-\lambda^q}}
\\
&\phantom{XXX}=\frac{1}{z+f(\infty)}\lim_{R\to\infty}\,\int_{|\zeta|=R}\,
\frac{f(\infty)-f(\zeta^{1/q})}{z+f(\zeta^{1/q})}\,
\frac{d\zeta}{\zeta-\lambda^q}\\
&\phantom{XXX}=0.
\end{align*} 
This proves the statement (\ref{R11+}).

The statement (\ref{R11++}) can be deduced from (\ref{R11+}) by replacing $f(\zeta)$ by $f(1/\zeta)$, $\lambda$ by $1/\lambda$ and $t$ by $1/t$.
\end{proof}

In order to estimate operator-valued integrals that will arise in Section \ref{pres}, we will need to estimate integrals such as the one in (\ref{R11+}) in specific ways.  Let $q>2$, $\lambda\in \Sigma_{\theta}$ where $\theta \in [0,\pi/q)$, and $z \in \Sigma_\psi$ where $\psi \in (0,\pi- \pi/q)$.  Using (\ref{Pest}), we can estimate as follows:
\begin{align}  \label{integ}
&\\
\lefteqn{\left| \int_0^\infty
\frac{\operatorname{Im}f(te^{i\pi/q})\,t^{q-1} }
{(z+f(te^{i\pi/q}))(z+f(te^{-i\pi/q}))(\lambda^q+t^q)} \,dt \right|} \nonumber\\
&\le \int_0^\infty
\frac{|\operatorname{Im}f(te^{i\pi/q})| \,t^{q-1}}
{|z+f(te^{i\pi/q})| \, |z+f(te^{-i\pi/q})|\,|\lambda^q+t^q|} \,dt \nonumber\\
& \le \frac{1}{\cos^2((\psi+\pi/q)/2) \cos^2(\pi/q) \cos(\theta q/2)} \int_0^\infty \frac{|\operatorname{Im}f(te^{i\pi/q})|}
{(|z|+f(t))^2 t} \,dt.  \nonumber
\end{align}
Here we are assuming existence of the integrals in appropriate senses.  We shall need that the final integral is not only finite but also that it satisfies a further estimate given in the following definition.

\def\E{\mathcal{E}}
\begin{defn} \label{defcong}
We shall say that $f \in \mathcal{E}$ if $f \in \NPp$ and, for each $\theta \in (0,\pi/2)$, there exists $\kappa_\theta(f)$ such that
\begin{equation} \label{tilde}
J_\theta(r;f) := \int_0^\infty \frac{|\operatorname{Im}f(te^{i\theta})|}
{(r+f(t))^2 t} \,dt \le  \frac{\kappa_\theta(f)}{r}, \qquad r>0.
\end{equation}

We shall say that $f \in \mathcal{E}_0$ if $f \in \NPp$ and, for each $\theta\in(0,\pi/2)$ and for some (or equivalently, all) $a\in(0,\infty)$, there exists $C_{a,\theta}$ such that
\[
\int_0^a \frac{|\operatorname{Im}f(te^{i\theta})|}
{(r+f(t))^2 t} \,dt  \le \frac{C_{a,\theta}}{r}, \qquad  r >0.
\]
 
Similarly, we say that $f \in \mathcal{E}_\infty$ if $f \in \NPp$ and, for each $\theta\in(0,\pi/2)$ and some/all $a\in(0,\infty)$, there exists $K_{a,\theta}$ such that
\[
\int_a^\infty \frac{|\operatorname{Im}f(te^{i\theta})|}
{(r+f(t))^2 t} \,dt  \le \frac{K_{a,\theta}}{r}, \qquad  r >0.
\]
\end{defn}

Note that $\mathcal{E} = \mathcal{E}_0 \cap \mathcal{E}_\infty$.  Moreover $\mathcal{E}_0$ and $\mathcal{E}_\infty$ are closed under sums.  This follows from
\[
\frac{|\operatorname{Im} \left(f(te^{i\theta}) + g(te^{i\theta})\right)|}{(r + f(t) + g(t))^2 }
\le \frac{|\operatorname{Im} f(te^{i\theta})|}{(r + f(t))^2} + 
\frac{|\operatorname{Im} g(te^{i\theta})|}{(r + g(t))^2},
\]
when $r>0$, $f(t)>0$ and $g(t)> 0$.

Simple changes of variable, and applications of (\ref{ts3}) and (\ref{ts4}), show that the following are equivalent:
\begin{enumerate}[\rm(i)]
\item $f \in \E_0$,
\item $1/f \in \E_0$,
\item $f(1/z) \in \E_\infty$,
\item $1/f(1/z) \in \E_\infty$.
\end{enumerate}
If $f \in \E_\varpi$ where $\varpi \in \{0,\infty\}$, and $\alpha\in(0,1)$, then $f(z^\alpha) \in \E_\varpi$.

It follows from the remarks above that if $f$ and $g$ are in $\E$, then so are $f+g$, $1/f$, $f(1/z)$, $1/f(1/z)$ and $f(z^\alpha)$ for $\alpha \in (0,1)$.

\begin{prop} \label{Elim}
If $f \in \mathcal{E}$, then $f(0+)$ and $f(\infty)$ exist in $[0,\infty]$.
\end{prop}

\begin{proof}
Consider the final formula (\ref{RR+}) where $z$ is fixed (for example, $z=R=1$).  As $\lambda \to 0+$, the first integral converges by the Dominated Convergence Theorem.  As in (\ref{integ}) the integrand of the second integral is dominated by the integrand in (\ref{tilde}) (for $r=1$) and the integral there is finite.  By the Dominated Convergence Theorem, the second integral converges to a limit as $\lambda\to0+$.  Hence
\[
\lim_{\lambda\to0+} (z+f(\lambda))^{-1}
\]
exists in $\C$, so $f(0+)$ exists.   The existence of $f(\infty)$ can now be deduced by considering $f(1/z)$.
\end{proof}

Since there exist $f \in \NPp$ for which $f(0+)$ and $f(\infty)$ do not exist, Proposition \ref{Elim} shows that $\mathcal{E} \ne \NPp$.  It will follow from Theorem \ref{ops} and Corollary \ref{TCbf13} that there exist $f \in \NPp$ such that $f(0+)$ and $f(\infty)$ both exist but $f \notin \mathcal{E}$.

In the remainder of this section, we will identify some conditions which imply that a given $f \in \NPp$ belongs to $\E$.  See Propositions \ref{sphere} and \ref{lemDE}.

Recall that the \emph{spherical derivative} $Sf$ of a holomorphic function $f$ on $\mathbb C_+$  is defined as
\[
(Sf)(z):=\frac{|f'(z)|}{1+|f(z)|^2}, \qquad z \in \mathbb C_+.
\]
It is a basic tool in geometric function theory; see \cite{Pom}, for example.  

Let $\mathcal{S}$ be the set of all $\NPp$-functions $f$ such that, for each $\theta \in (0,\pi/2)$, there exists $C_\theta$ such that 
\[
\int_{\Sigma_\theta} \,(S (f/r) )(z)\,\frac{dm(z)}{|z|}  \le C_\theta, \qquad r>0,
\]
where $dm$ is the area measure.

\begin{prop} \label{sphere}
Let $f\in \mathcal{NP}_{+}$, $r>0$, and $\theta\in (0,\pi/2)$.   Then
\begin{align*}
J_\theta(r;f) &\le \frac{1}{2 \cos^2\theta}\int_{-\theta}^{\theta}\int_0^\infty
\,\frac{|f'(te^{is})|}{r^2+ |f(te^{is})|^2}\,dt \,ds  \\
&= \frac{1}{2 r \cos^2\theta } \int_{\Sigma_\theta} \,(S (f/r) )(z)\,\frac{dm(z)}{|z|}.
\end{align*}
Hence, $\mathcal{S} \subset \mathcal{E}$.
\end{prop}

\begin{proof}
Since
\[
2i \operatorname{Im}f(te^{i\theta})= f(te^{i\theta}) - f(te^{- i\theta}) =
it\int_{-\theta}^{\theta} e^{is} f'(te^{is})\,ds,
\]
we have 
\begin{equation}  \label{imfest}
2|\operatorname{Im}f(te^{i\theta})|\le
t\int_{-\theta}^{\theta} |f'(te^{is})|\,ds,\quad t>0.
\end{equation}
Using (\ref{ts4}), we have
\begin{align*}
2J_\theta(r;f)\le &\int_0^\infty
\int_{-\theta}^{\theta} |f'(te^{is})|\,ds
\,\frac{dt}{(r+f(t))^2}\\
=&\int_{-\theta}^{\theta}\int_0^\infty
\,\frac{ |f'(te^{is})|}{(r+f(t))^2}\,dt \,ds\\
\le&\int_{-\theta}^{\theta}\int_0^\infty
\,\frac{|f'(te^{is})|}{(r+\cos s |f(te^{is})|)^2}\,dt \,ds\\
\le&\frac{1}{\cos^2\theta}\int_{-\theta}^{\theta}\int_0^\infty
\,\frac{|f'(te^{is})|}{r^2+ |f(te^{is})|^2}\,dt \,ds.  \\
=&\frac{1}{r \cos^2\theta} \int_{\Sigma_\theta} \,(S (f/r) )(z)\,\frac{dm(z)}{|z|}. \qedhere
\end{align*}
\end{proof}

Our second class of functions included in $\E$ involves pointwise conditions comparing the imaginary part of $f(z)$ in $\C_+$ with the derivative of $f(t)$ on $(0,\infty)$.  These conditions may be more intuitive than Definition \ref{defcong}.

\begin{defn} \label{defD*}
We introduce here four conditions on $f \in \NPp$, labelled as $\mathcal{D}_\varpi^\pm$, where $\varpi \in \{0,\infty\}$.   Each condition is to be read as saying that $f \in \mathcal{D}_\varpi^\pm$ if, for each $\theta \in (0,\pi/2)$, there exist $a \in (0,\infty]$, $b\in(0,\infty)$ and $c\in(0,\infty)$ (which may depend on $\theta$ and $f$) such that the specified properties hold:
\begin{itemize}
\item $f \in \mathcal{D}_0^+$ if $f$ is increasing on $(0,a/b)$ and 
\[
|\operatorname{Im}f(te^{i\theta})|\le c t f'(bt),\quad t\in (0,a).
\]
\item $f \in \mathcal{D}_0^-$ if $f$ is decreasing on $(0,a/b)$, and 
\[
|\operatorname{Im}f(te^{i\theta})|\le -c t f'(bt),\quad t\in (0,a).
\]
\item $f \in \mathcal{D}_\infty^+$ if $f$ is increasing on $(a/b,\infty)$, and 
\[
|\operatorname{Im}f(te^{i\theta})|\le c t f'(bt),\quad t\in (a,\infty).
\]
\item $f \in \mathcal{D}_\infty^-$ if $f$ is decreasing on $(a/b,\infty)$, and 
\[
|\operatorname{Im}f(te^{i\theta})|\le -c t f'(bt),\quad t\in (a,\infty).
\]
\end{itemize}
\end{defn}

\begin{example} \label{compsupp}
It is easy to see that the $\NPp$-functions $z(s^2+z^2)^{-1}$ belong to $\mathcal{D}_0^+ \cap \mathcal{D}_\infty^-$ for $s \in (0,\infty)$, and that the constants $a,b,c$ in Definition \ref{defD*} can be chosen uniformly for $s$ in any compact subset of $(0,\infty)$.  Hence if $\nu$ is a finite positive measure with compact support in $(0,\infty)$ and
\[
g(z) = z \int_0^\infty \frac{\nu(ds)}{s^2+z^2},
\]
then $g \in \mathcal{D}_0^+ \cap \mathcal{D}_\infty^-$.  In particular, for any sequence $(s_k)_{k\ge1}$ from a compact subset of $(0,\infty)$ and any sequence $(a_k)_{\ge1}$ with $a_k>0$ and $\sum_{k=1}^\infty a_k <\infty$, the function $g$ defined by
\[
g(z) = \sum_{k=1}^\infty a_k \frac{z}{z^2 + s_k^2},
\] 
belongs to $\mathcal{D}_0^+ \cap \mathcal{D}_\infty^-$.  This fact will be used in Example \ref{nonhp}.
\end{example}

Using Proposition \ref{pnpp}(3) it is easy to check that:
\[
f \in \mathcal{D}_{0}^+ \iff 1/f \in \mathcal{D}_{0}^- \iff f(1/z) \in \mathcal{D}_{\infty}^- \iff 1/f(1/z) \in \mathcal{D}_{\infty}^+.
\]

Let $f \in \mathcal{D}_0^+$.  A very simple calculation shows that $f(z^\alpha) \in \mathcal{D}_0^+$ when $\alpha \in (0,1)$.  For compositions of the form $g \circ f$, we have the following result.

\begin{prop} \label{Dcomp}
Let $f \in \mathcal{D}_0^+$, and let $g$ be a holomorphic function on $\mathbb C_+$ such that $g \circ f \in \mathcal{NP}_+$ and $g$ maps $(0,\infty)$ into $(0,\infty)$. 
\begin{enumerate}[\rm a)]
\item  \label{Dca} Assume that $f(0+)>0$, and $g'(f(0+))>0$.  Then $g \circ f \in \mathcal{D}_0^+$.
\item \label{Dcb} Assume that $f(0+)=0$ and there exists $r>0$ such that $g$ is strictly increasing on $(0,r)$, and, for each $\theta>0$, there exists $C_\theta>0$ such that
\begin{enumerate}[\rm(i)]
\item \label{Dci}$|g'(te^{i\psi})| \le C_\theta g'(t), \qquad  t \in (0,r), \; -\theta < \psi < \theta$;
\item \label{Dcii} $g'(t_1) \le C_\theta g'(t_2), \qquad t_1,t_2 \in (0,r), \; \frac12 \le t_1/t_2 \le 2$.
\end{enumerate}
Then $g \circ f \in \mathcal{D}_0^+$.
\end{enumerate}
\end{prop}

\begin{proof}  We give the proof for the case \ref{Dcb}).  Note that assumption (b)(ii) can be iterated to give
\begin{equation} \label{gprime}
g'(t_1) \le C_\theta^{\log_2 \sigma + 1} g'(t_2), \qquad t_1,t_2 \in (0,r), \; 1/\sigma \le t_1/t_2 \le \sigma, \; \sigma>1.
\end{equation}
Let $\theta \in (0,\pi/2)$, and take $a,b,c$ as in Definition \ref{defD*} for $f \in \mathcal{D}_0^+$.  Let $f(te^{i\theta}) = \rho e^{i\psi}$, so $\rho$ and $\psi$ depend on $t$ and $|\psi| \le \theta$ by Proposition \ref{pnpp}(\ref{pnpp1}).   For $t>0$ sufficiently small, we have
\begin{equation} \label{imf}
\frac{2}{\pi} \rho |\psi| \le  \rho |\sin\psi| = |\operatorname{Im} f(te^{i\theta})| \le ct f'(bt),
\end{equation}
since $\sin s/s$ is decreasing on $[0,\pi/2]$.  By Proposition \ref{pnpp}(\ref
{pnpp3}), there is a constant $\sigma>1$, depending on $\theta$ and $b$, such that
\[
\frac{f(bt)}{\sigma} \le \rho = |f(te^{i\psi})| \le \sigma f(bt), \qquad t>0.
\]
For $t>0$ sufficiently small, using (\ref{imfest}), the assumption (b)(i), (\ref{imf}) and (\ref{gprime}) we have
\begin{align*}
\big|\operatorname{Im} g\big(f(te^{i\theta})\big)\big| &\le  \rho \int_0^\psi \left| g'(\rho e^{is}) \right| \, ds\\
&\le C_\theta \rho |\psi| g'(\rho)  \\
& \le C_\theta \frac{\pi}{2} ct f'(bt) C_\theta^{\log_2\sigma + 1} g'(f(bt)) \\
&= \frac{\pi cC_\theta^{\log_2\sigma+2}}{2} t (g \circ f)'(bt).
\end{align*}

In the case \ref{Dca}), the proof is simpler, using (\ref{imf}) and continuity of $g$ and $g'$ near $f(0+)$. 
\end{proof}

The assumptions of Proposition \ref{Dcomp}(\ref{Dcb}) are satisfied when $g(z) = z^\beta$ for $\beta>1$.  Hence if $f \in \mathcal{D}_0^+$, then the following are also in $\mathcal{D}_0^+$:
\begin{enumerate}[(i)]
\item  $f(z)^\beta$  if  $\beta>0$ and $f(z)^\beta \in \NPp$;
\item  $f(z)^\alpha$ and $f(z^\alpha)^{1/\alpha}$  when $\alpha \in (0,1)$.  
\end{enumerate}
Corresponding statements hold for $\mathcal{D}_0^-$, $\mathcal{D}_\infty^+$ and  $\mathcal{D}_\infty^-$.

In Definition \ref{defD*} the value of $b$ may depend on $f$, so the sum of two functions in any one of the four  classes may not be in the class.  We make the following definition which allows for sums.

\begin{defn} \label{defD}
For $\varpi \in \{0,\infty\}$, let $\mathcal{D}_\varpi$ be the intersection of all subsets $\mathcal{G}$ of $\NPp$ such that
\begin{enumerate}[\rm(i)]
\item $\mathcal{D}_\varpi^+ \subset \mathcal{G}$,
\item If $f \in \mathcal{G}$, then $1/f \in \mathcal{G}$,
\item If $f_1,f_2 \in \mathcal{G}$, then $f_1 + f_2 \in \mathcal{G}$.
\end{enumerate}
Let $\mathcal{D} = \mathcal{D}_0 \cap \mathcal{D}_\infty$.
\end{defn}

In other words, $\mathcal{D}_0$ contains those functions $f$ which can be obtained from a finite number of functions in $\mathcal{D}_0^+$ by taking a finite number of sums and reciprocals in some order.   In the definition, $\mathcal{D}_0^+$ could be replaced by $\mathcal{D}_0^-$.  Moreover, $\mathcal{D}_\infty = \{f \in \NPp: f(1/z) \in \mathcal{D}_0\}$.

\begin{prop}  \label{lemDE}
Let $\varpi \in \{0,\infty\}$.  Then $\mathcal{D}_\varpi \subset \E_\varpi$.  Hence $\mathcal{D} \subset \mathcal{E}$.
\end{prop}

\begin{proof}
Let $f \in \mathcal{D}_0^+$, $a,b,c$ be as in Definition \ref{defD*}, $a' = a \min(1,1/b)$ and $\rho = \max(b,1/b)$.  Let $\theta \in (0,\pi/2)$ and $r>0$.  Since $f$ is increasing on $(0,a)$ and (\ref{ts3}) holds, we have
\[
f(bt) \le \max(b,1)f(t), \qquad t \in (0,a').
\]
Then
\begin{equation} \label{Ctheta}
\int_0^{a'} \frac{|\operatorname{Im}f(te^{i\theta})|}{(r+f(t))^2 t} \,dt
\le c \max(1,b^2) \int_0^a \frac{f'(bt)}{(r+f(bt))^2} \, dt \le \frac{c \max(1,b^2)}{ b(r + f(0+))} \le \frac{c\rho}{ r}.
\end{equation}
So $f\in \E_0$.  Since $\E_0$ is closed under reciprocals and sums, it follows from Definition \ref{defD} that $\mathcal{D}_0 \subset \E_0$.

A similar argument applies for $\varpi = \infty$.  Then
\[
\mathcal{D} = \mathcal{D}_0 \cap \mathcal{D}_\infty \subset \E_0 \cap \E_\infty = \E. \qedhere
\]
\end{proof}

We shall give examples of some broad classes of functions in $\mathcal{D}$ in Section \ref{BF}.   Here we mention one other class of examples which slightly extends \cite[Lemma 3.7, Appendix A]{GT2}.  Let
\[
h_n(z):=\left(\frac{1-z}{1+z}\right)^n,\quad n\in \mathbb N, \; z \in \C_+.
\]

\begin{prop}
If
\begin{equation}\label{L1}
h(z)=1-\sum_{n=0}^\infty c_n h_n(z),\quad
c_n \in \mathbb R, \quad \sum_{n=0}^\infty |c_n|\le 1,
\end{equation}
then $h\in \mathcal{D}$.
\end{prop}

\begin{proof}
From
\[
|h_n(z)|\le 1,\quad z\in  \mathbb C_{+}, \quad h_n: (0,\infty)\to \mathbb R,
\]
it follows that $h\in \mathcal{NP}_{+}$.  Moreover $h_n$ is decreasing on $(0,1)$.

By \cite[Lemma 9.4]{GT2},
 for all $\theta\in (0,\pi/2)$ there exists
$b\in(0,1)$ such that for every $n\in \mathbb N$
\[
|\operatorname{Im} h_n(te^{i\theta})|\le -\frac{\pi}{2}t h_n'(bt),\quad t\in (0,1).
\]
Let 
\[
g_1(z)=
a+\sum_{c_n\ge 0}c_n(1-h_n(z)),\quad
g_2(z)=\sum_{c_n<0}c_n(1-h_n(z)), 
\quad a=1-\sum_{n=0}^\infty c_n.
\]
Then
\[
h=g_1+g_2,\quad g_1\in \mathcal{D}_0^{+},\quad g_2\in \mathcal{D}_0^{-}.
\]
So $h \in \mathcal{D}_0$.  This applies to any $h$ of the form (\ref{L1}).

Since
\[
\tilde{h}(z):=h(1/z)=
1-\sum_{n=0}^\infty (-1)^nc_n h_n(z)
\]
is also of the form (\ref{L1}), we have
\[
\tilde{h}\in \mathcal{D}_0.
\]
Then
\[
h(z)=\tilde{h}(1/z)\in \mathcal{D}_\infty. \qedhere
\]
\end{proof}

\section{Function classes contained in $\mathcal{D}$}  \label{BF}

We start by recalling  definitions of some established classes of functions.

\begin{defn}
A smooth function $g: (0,\infty) \to (0,\infty)$ is \emph{completely monotone} if $(-1)^n g^{(n)}(t) \ge 0$ for all $t>0$ and $n\in\N$.
\end{defn}

By Bernstein's Theorem \cite[Theorem 1.4]{SchilSonVon2010} $g$ is completely monotone if and only if there exists a positive Laplace transformable measure $\nu$ on $[0,\infty)$ such that
\begin{equation} \label{defCM}
g(t) = \int_0^\infty e^{-t s} \, \nu(ds), \qquad t>0.
\end{equation}
Then $g$ may be extended to a holomorphic function from $\C_+$ to $\C$, and we shall identify $g$ with this extension.  We have
\begin{equation} \label{CMinq}
|g(z)| \le \int_0^\infty e^{-\operatorname{Re} z s} \, \nu(ds) = g(\operatorname{Re}z), \qquad z \in \C_+.
\end{equation}
We shall denote the set of all completely monotone functions by $\mathcal{CM}$.

\begin{defn}
A smooth function $f : (0,\infty)\mapsto (0,\infty)$ is called a \emph{Bernstein function}
if its derivative $f'$ is completely monotone.
\end{defn}
By \cite[Theorem 3.2]{SchilSonVon2010},
$f$ is Bernstein if and only if there exist  $a, b\geq 0$
and a positive measure $\mu$ on $(0,\infty)$ satisfying
\begin{equation*}
\int_{0+}^\infty\frac{s}{1+s}\,\mu(d{s})<\infty \label{mu}
\end{equation*}
such that
\begin{equation}\label{defBF}
f(t)=a+bt+\int_{0+}^\infty (1-e^{-ts})\mu(d{s}), \qquad t>0.
\end{equation}
The formula \eqref{defBF} is called the \emph{L\'evy-Khintchine representation} of $f$.  The triple $(a,b,\mu)$ is defined uniquely and is called the \emph{L\'evy triple} of $\psi$. 
Standard examples of Bernstein functions include $z^\alpha$ for  $\alpha \in [0,1]$,  $\log(1+z)$ and $1 - e^{-z}$.

Every Bernstein function extends to a holomorphic function from $\mathbb C_+$ to $\C_+$ given by the formula (\ref{defBF}), and we will identify Bernstein functions with their holomorphic extensions to $\C_+$.  The set of all Bernstein functions will be denoted by $\mathcal{BF}$ and it will be identified with a subset of $\NPp$.    

In the next proposition we summarise some known properties of Bernstein functions similar to those in Proposition \ref{pnpp}.  For these and other properties of Bernstein functions, see \cite[Chapter 3]{SchilSonVon2010} or \cite[Section 3.9]{Jacob}.

\begin{prop} \label{pbf}
Let $f \in \mathcal{BF}$.
\begin{enumerate}[\rm1.]
\item \label{pbf1} $f$ maps $\Sigma_\theta$ into itself for all $\theta \in (0,\pi/2)$.
\item \label{pbf2}  If $\alpha \in (0,1/2]$ and $h_\alpha(z) = f(z^\alpha)^{1/\alpha}$ for $z \in \C_+$, then $h_\alpha \in \mathcal{BF}$.
\item \label{pbf5}  If $g \in \mathcal{CM}$, then $g \circ f \in \mathcal{CM}$.  In particular, $1/f \in \mathcal{CM} \cap \NPp$.
\item \label{pbf6}  If $h \in \mathcal{BF}$, then $h \circ f \in \mathcal{BF}$.
\item \label{pbf3} For all  $\theta\in(-\pi/2,\pi/2)$ and $t>0$,
\begin{equation*}\label{R09}
f(t) \cos\theta \le f(t\cos\theta) \le |f(te^{i\theta})| \le  \kappa_\theta f(t) \le \frac{\kappa_\theta}{\cos\theta} f(t\cos\theta), 
\end{equation*}
where
\[
\kappa_\theta = \min \left(\frac{2e}{e-1}, \frac{1}{\cos\theta}\right) .
\]
\item \label{pbf4} For all  $\theta\in (0,\pi/2)$,
\begin{equation}\label{in20}
\lim_{z \in \Sigma_\theta,|z|\to\infty} f(z) = f(\infty).
\end{equation}
\end{enumerate}
\end{prop}

\begin{proof}
(\ref{pbf1}) and (\ref{pbf4}) are immediate from the corresponding parts of Proposition \ref{pnpp}, (\ref{pbf2}), and (\ref{pbf5}) and (\ref{pbf6}) are  shown in \cite[Proposition A.2]{GT},  \cite[Theorem 3.7]{SchilSonVon2010}, and \cite[Corollary 3.8]{SchilSonVon2010}, respectively.   For (\ref{pbf3}), the first and fourth inequalities follow from (\ref{ts3}), and the third inequality is a combination of (\ref{ts4}) and \cite[Proposition 2.3]{GT3}.  The second inequality follows from (\ref{CMinq}) applied to $1/f$.
\end{proof}

It is an open question whether Proposition \ref{pbf}(\ref{pbf2}) holds for $\alpha \in (1/2,1)$.  If $f : (0,\infty) \to (0,\infty)$ and $g \circ f \in \mathcal{CM}$ for all $g \in \mathcal{CM}$ then $f \in \mathcal{BF}$ \cite[Theorem 3.7]{SchilSonVon2010}.  The following gives examples where $g \in \mathcal{CM} \cap \NPp$ but $1/g \notin \mathcal{BF}$.

\begin{example} \label{exnew}
If $t\ge0$ and
\[
g_t(z) = \frac{z+2t}{(z+t)^2} = \frac{1}{z+t} + \frac{t}{(z+t)^2},
\]
then $g_t \in \mathcal{CM} \cap \NPp $.  However, for $t>0$, 
\[
\frac{1}{g_t(z)} = z + \frac {t^2}{z+2t} \notin \mathcal{BF},
\]
because its second derivative is positive.  If $\mu$ is any finite positive measure on $[0,\infty)$ and 
\[
\int_0^\infty \frac{\mu(dt)}{1+t} < \infty,   \quad G(z) = \int_0^\infty g_t(z) \, \mu(dt),
\]
then $G \in  \mathcal{CM} \cap \NPp$.
\end{example}

In the following results we will show that $\mathcal{D}$ contains all functions in the following classes:
\begin{enumerate}[\rm (D1)] \label{dcl}
\item \label{dcl1} $\mathcal{CM} \cap \NPp$ (Proposition \ref{cong}\ref{conga}),
\item $\mathcal{BF}$ (Proposition \ref{cong}\ref{congb}),
\item Functions in $\NPp$ which are products of finitely many functions in $\mathcal{BF}$ (Corollary \ref{Pp1}),
\item \label{dcl4} Functions of the form $f_1(z^{\alpha_1})^{\beta_1} \dots f_n(z^{\alpha_n})^{\beta_n}$, where $f_j \in \mathcal{BF}$, $0 < \alpha_j \le 1$, $0 < \beta_j \le 1$ and $\sum_{j=1}^n \alpha_j \beta_j \le 1$ (Corollary \ref{Pp2}).
\end{enumerate} 
See Examples \ref{nonbf} and \ref{pans} for some illustrations of  the class (D\ref{dcl4}).

\begin{prop} \label{cong}
\begin{enumerate} [\rm a)] 
\item \label{conga} Let $g \in \mathcal{CM} \cap \NPp$.  Then
\[
|\operatorname{Im}g(te^{i\theta})| \le -|\sin\theta| \, t g'(t\cos\theta), \qquad t>0, \; \theta \in (-\pi/2,\pi/2).
\]
Hence $g \in \mathcal{D}$, and in particular $g \in \E$.

\item \label{congb} Let $f \in \mathcal{BF}$.  Then $f \in \mathcal{S} \cap \mathcal{D}$.  In particular,
\begin{equation} \label{BFD}
|\operatorname{Im} f(te^{i\theta})| \le 
|\sin\theta|\,t f'(t\cos\theta), \qquad t>0, \; \theta \in (-\pi/2,\pi/2),
\end{equation}
and $f \in \E$.
\end{enumerate}
\end{prop}

\begin{proof}
Let $g\in \mathcal{CM}$.  From (\ref{defCM}), we have
\[
g'(t)=-\int_{0+}^\infty se^{-ts}\,\nu(ds)<0,\quad t>0,
\]
and for $\theta\in (-\pi/2,\pi/2)$ and $t>0$ it follows that
\begin{align*}
|\operatorname{Im} g(te^{i\theta})| &=\left|\int_{0+}^\infty \operatorname{Im}\,e^{-ste^{i\theta}}\,\nu(ds)\right| 
\le \int_{0+}^\infty e^{-st\cos\theta}|\sin(ts\sin\theta)|\,\nu(ds)  \\
&\le |\sin\theta| t \int_{0+}^\infty s e^{-st\cos\theta}\,\nu(ds)
\le -|\sin\theta| t g'(t\cos\theta).
\end{align*}
So, if we suppose additionally that $g\in \mathcal{NP}_{+}$ then 
\[
g\in \mathcal{D}_0^{-} \cap \mathcal{D}_\infty^- \subset \mathcal{D}.
\]
By Proposition \ref{lemDE}, $g\in\E$.

Let $f \in \mathcal{BF}$.  Then $1/f \in \mathcal{CM} \cap \NPp \subset \mathcal{D}$, so $f \in \mathcal{D} \subset \E$.  Moreover, $f' \in \mathcal{CM}$, so $|f'(te^{i\theta})| \le f'(t\cos\theta)$, by (\ref{CMinq}).  Using also Proposition \ref{pbf}(\ref{pbf3}), we have
\begin{align*}
\int_0^\infty \frac{|f'(te^{i\theta})|} {r^2 + |f(te^{i\theta})|^2} \, dt
&\le \int_0^\infty \frac{f'(t\cos\theta)} {r^2 + f(t\cos\theta)^2} \, dt \\
&\le \frac{1}{\cos\theta} \int_{0}^{\infty} \frac{d\tau}{r^2+\tau^2} 
= \frac{\pi}{2 r \cos \theta}. 
\end{align*}
By Proposition \ref{sphere}, $f \in \mathcal{S} \subset \mathcal{E}$.

By (\ref{defBF}) we have
\[
f'(z)=b+\int_{0+}^\infty s e^{-sz}\,\mu(ds),\quad z\in \C_{+},
\]
and
\[
\operatorname{Im}f(te^{i\theta})
=bt\sin\theta+\int_{0+}^\infty  e^{-st\cos\theta}
\sin(ts\sin\theta)\,\mu(ds).
\]
Hence, using 
$|\sin \tau|\le |\tau|$,  we obtain that
\begin{align*}
|\operatorname{Im}f(te^{i\theta})| &\le bt|\sin\theta|+
t|\sin \theta|\int_{0+}^\infty  se^{-st\cos\theta}
\,\mu(ds)
\\
&=|\sin\theta|\,t f'(t\cos\theta). \nonumber \qedhere
\end{align*}
\end{proof}

By Example \ref{exnew} and Proposition \ref{cong}\ref{conga}),  if 
\[
g(z) = \sum_{k=1}^\infty a_k \frac{z+2t_k}{(z+t_k)^2}\,,
\]
where $a_k,t_k>0$ and $\sum_{k=1}^\infty {a_k}(1+t_k)^{-1} < \infty$, then $g \in \mathcal{D}$.

\begin{prop}\label{Pp1A}
Let $f_j$, $j=1,\dots,n$, be holomorphic on $\mathbb C_{+}$ and map $(0,\infty)$ to $(0,\infty)$. Suppose that for every $\theta\in (-\pi/2,\pi/2)$ there exist  $b_\theta,c_\theta,d_\theta>0$ such that 
\begin{equation}\label{R1+}
|f_j(te^{i\theta})|\le d_{\theta}f_j(b_\theta t),\qquad t>0, \, j=1,\dots,n
\end{equation}
and
\begin{equation}\label{R2+}
|\operatorname{Im} f_j(te^{i\theta})|\le c_{\theta} t f_j'(b_\theta t),\qquad t>0, j=1,\dots,n.
\end{equation}
Let
\[
F_n=f_1\cdots f_n.
\]
Then
\begin{equation}\label{Im1}
|F_n(te^{i\theta})|\le d_\theta^n F_n(b_\theta t),\quad |\operatorname{Im}F_n(te^{i\theta})|\le c_\theta d_\theta^{n-1} t F_n'(b_\theta t),\qquad t>0.
\end{equation}
\end{prop}

\begin{proof}
We proceed by induction.  If $n=1$ then (\ref{Im1}) reduces to  (\ref{R1+}) and (\ref{R2+}).
Suppose that  the inequalities (\ref{Im1}) hold for fixed $n\in \mathbb N$.
Then we have
\begin{align*}
|\operatorname{Im}F_{n+1}(te^{i\theta})|
&= |\operatorname{Im}F_n(te^{i\theta})f_{n+1}(te^{i\theta})|\\
&\le |\operatorname{Im}F_n(te^{i\theta})|\cdot |f_{n+1}(te^{i\theta})|
+ |F_n(te^{i\theta})| \cdot |\operatorname{Im}f_{n+1}(te^{i\theta})|\\
&\le t c_\theta  d_\theta^{n-1}  F_n'(b_\theta t) d_\theta f_{n+1}(b_\theta t)+ 
d_\theta^n F_n(b_\theta t) tc_\theta f_{n+1}'(b_\theta t)\\
&= t c_\theta d_\theta^n (F_n\cdot f_{n+1})'(b_\theta t)\\
&= t c_\theta d_\theta^n F_{n+1}'(b_\theta t).
\end{align*}
The rest is clear.
\end{proof}

\begin{cor}\label{Pp1}
Let $\{f_j: j = 1,\dots, n \} \subset \mathcal{BF}$, and let
\[
F_n=f_1\cdots f_n.
\]
Then 
\[
|\operatorname{Im} F_n(te^{i\theta})|\le \frac{|\sin\theta|}{\left(\cos\theta\right)^{2(n-1)}} t F_n'(t\cos\theta),\qquad t>0, \; \theta \in (-\pi/2,\pi/2).
\]
In particular, if $F_n\in \mathcal{NP}_+$, then $F_n \in \mathcal{D}_0^{+} \cap \mathcal{D}_\infty^+ \subset \mathcal{D}$.
\end{cor}

\begin{proof}  Let $t>0$ and $\theta \in (-\pi/2,\pi/2)$.  By (\ref{ts4}) and (\ref{ts3}) with $r=t\cos\theta$,
\begin{equation*}
|f_j(te^{i\theta})|\le \frac{f_j(t)}{\cos\theta} \le  \frac{f_j(t\cos\theta)}{\cos^2\theta},
\end{equation*}
By (\ref{BFD}),
\[
|\operatorname{Im}\,f_j(te^{i\theta})|\le |\sin\theta| t f_j'(t\cos\theta).
\]
Now the statement is a direct consequence of Proposition \ref{Pp1A}.
\end{proof}

\begin{cor}\label{Pp2} If 
$\{f_j: j =1,\dots,n \} \subset \mathcal{BF}$, and
$\alpha_j,\beta_j$, $j =1,\dots,n$, are such that  
\[
0 < \alpha_j \le 1, \quad 0 < \beta_j \le 1, \quad \sum_{j=1}^n\alpha_j\beta_j\le1,
\]
then 
\begin{equation}\label{Exq}
\tilde{F}_n(z) :=f_1(z^{\alpha_1})^{\beta_1}\cdots f_n(z^{\alpha_n})^{\beta_n} 
\end{equation}
belongs to $\mathcal{D}$.
\end{cor}

\begin{proof}  Each function $f_j(z^{\alpha_j})^{\beta_j}$ is a Bernstein function, and it maps $\C_+$ into $\Sigma_{\pi\alpha_j\beta_j/2}$ by Proposition \ref{pnpp}(\ref{pnpp1}).  Hence $\tilde F_n$ maps $\C_+$ into $\C_+$.  Moreover $\tilde F_n$ maps $(0,\infty)$ to $(0,\infty)$, so $\tilde F_n \in \NPp$.  By Corollary \ref{Pp1}, $\tilde F_n \in \mathcal{D}$. 
\end{proof} 

When $\beta_j = 1$ for $j=1,\dots,n$, the function $\tilde F_n$ of Corollary \ref{Pp2} belongs to $\mathcal{BF}$.  This is shown in \cite[Corollary 3.8(vi)]{SchilSonVon2010} for $n=2$ and the general case follows by a simple induction.   The following example shows that some functions of the form (\ref{Exq}) are not Bernstein functions.

\begin{example}  \label{nonbf}  The functions $z$ and $1-e^{-z}$ are both Bernstein functions.  Moreover 
\[
\tilde{F}(z):=z^{1/2}(1-e^{-z})^{1/2}
\]
is not a Bernstein function \cite[Example 5.7]{BGT15}, but $\tilde F \in \mathcal{D}$ by Corollary \ref{Pp2}. 

For $\alpha \in (0,1)$, let $\tilde G_\alpha(z) = \tilde F(z^\alpha)$.  Since $\mathcal{BF}$ is closed under pointwise limits \cite[Corollary 3.8(ii)]{SchilSonVon2010}, there exists $\alpha \in (0,1)$ such that $\tilde G_\alpha \notin \mathcal{BF}$, but $\tilde G_\alpha \in \mathcal{D}$ by Corollary \ref{Pp2}.
\end{example}

\begin{example} \label{pans}
If  $f\in \mathcal{BF}$ and $\alpha\in (0,1)$ then 
\[
g_\alpha(z):=[f(z^\alpha)]^{1/\alpha}
\]
belongs to $\mathcal{NP}_{+}$.   If $\alpha \in (0,1/2]$  then $g_\alpha\in \mathcal{BF}$
by Proposition \ref{pbf}(\ref{pbf2}).   Although it is an open question whether $g_\alpha$ is necessarily in $\mathcal{BF}$ when $\alpha \in (1/2,1)$, we can show that $g_\alpha \in \mathcal{D}$.  For that case, let
\begin{gather*}
f_1 = f_2 = f, \\
\alpha_1=\alpha_2=\alpha, \quad \beta_1= \beta_2 = \frac{1}{2\alpha} \in (\tfrac12,1).
\end{gather*}
Then $g_\alpha = \tilde F_2$ as in (\ref{Exq}).  Hence it follows from Corollary \ref{Pp2} that $g_\alpha \in \mathcal{D}$.
\end{example}


It is often appropriate to restrict attention to the
subclass of Bernstein functions formed by the complete Bernstein
functions. This class has a rich structure which makes it especially
useful in applications.  A Bernstein function $\varphi$ is said to be
{\it a complete Bernstein function} if the measure $\mu$ in its
L\'evy-Khintchine representation (\ref{defBF}) has a completely
monotone density with respect to Lebesgue measure.
The set of all complete Bernstein functions will be denoted by $\mathcal{CBF}$.

The class of complete Bernstein functions allows a number of characterisations.
For example (see
\cite[Theorem 6.2]{SchilSonVon2010}),
a function $\varphi : (0,\infty)\mapsto (0,\infty)$
belongs to $\mathcal{CBF}$ if and only if $\varphi$ admits a representation  
which is given by
\begin{equation}\label{Cbf}
\varphi(t)=a+bt+\int_{0+}^\infty
\frac{t}{t+s} \,\nu(ds) \qquad t>0,
\end{equation}
where $a,b\ge 0$ are non-negative constants and $\nu$ is a positive measure
on $(0,\infty)$ such that
\[
\int_{0+}^\infty \frac{\nu(ds)}{1+s}<\infty.
\]
The triple $(a,b,\nu)$ is uniquely determined by these properties, and it is called the \emph{Stieltjes representation} of $\varphi$.  Then $\varphi$ has a holomorphic extension to $\C\setminus (-\infty,0]$ given by (\ref{Cbf}), and we shall identify $\varphi$ with this extension.  Note that $\varphi: (0,\infty) \to (0,\infty)$ is increasing, and it maps the upper and lower half-planes $\C^+$ and $\C^-$ into themselves.   If $\psi_2(z) = \varphi(z^{2})^{1/2}$ for $z\in\C_+$, then $\psi_2 \in \NPp$.   Consequently some properties of complete Bernstein functions can be readily deduced from Proposition \ref{pnpp}, although proofs from the Stieltjes representation may sometimes be more direct.

\begin{prop} \label{pcbf}
Let $\varphi \in \mathcal{CBF}$.
\begin{enumerate}[\rm1.]
\item \label{pcbf1} $\varphi$ maps $\Sigma_\theta$ into itself for all $\theta \in (0,\pi)$.
\item \label{pcbf2}  If $\alpha \in (0,1)$ and $\psi_\alpha(z) = \varphi(z^\alpha)^{1/\alpha}$ for $z \in \C \setminus (-\infty,0]$, then $\psi_\alpha \in \mathcal{CBF}$.
\item \label{pcbf3} For all  $\theta\in(-\pi,\pi)$ and $t>0$,
\begin{equation}\label{R10}
\varphi(t) \cos(\theta/2) \le |\varphi(te^{i\theta})|\le \frac{\varphi(t)}{\cos(\theta/2)}.
\end{equation}
\item \label{pcbf4} For all  $\theta\in (0,\pi)$,
\begin{equation*}
\lim_{z \in \Sigma_\theta,|z|\to\infty} \varphi(z) = \varphi(\infty).
\end{equation*}
\end{enumerate}
\end{prop}

\begin{proof}  (\ref{pcbf1})  and (\ref{pcbf4}) can be deduced by applying the corresponding parts of Proposition \ref{pnpp} to $\psi_2(z) = \varphi(z^2)^{1/2}$.  Moreover, (\ref{pcbf1}) and (\ref{pcbf3}) can be deduced from the Stieltjes representation (see \cite[Corollary 6.6]{SchilSonVon2010} for (\ref{pcbf1}), and \cite[Proposition 2.4]{BCT} for (\ref{pcbf3})).  For (\ref{pcbf2}), observe that $\psi_\alpha$ maps $\C^+$ to $\C^+$, $(0,\infty)$ to $(0,\infty)$, and $f(0+)$ exists and is real.  This implies that $\psi_\alpha \in \mathcal{CBF}$ \cite[Theorem 6.2]{SchilSonVon2010}.
\end{proof}

\begin{lemma}\label{phi}
Let $\varphi\in \mathcal{CBF}$. Then for any $\theta\in(0,\pi)$ and $t>0$,
\begin{equation}\label{R2}
|\operatorname{Im}\varphi(te^{i\theta})|\le
2\tan(\theta/2) \,t \varphi'(t).
\end{equation}
\end{lemma}

\begin{proof}
By  (\ref{Cbf}), (\ref{ts1}) and the representation
\[
\varphi'(t)=b +\int_{0+}^\infty \frac{s\,\sigma(ds)}{(t+s)^2},\quad t>0,
\]
it follows that
\begin{align*}
\operatorname{Im}\varphi(te^{i\theta}) &=
bt\sin\theta+\int_{0+}^\infty
\operatorname{Im}\frac{t e^{i\theta}}{te^{i\theta}+s}\,\sigma(ds)
=t\sin\theta
\left(b+
\int_{0+}^\infty\frac{s\,\sigma(ds)}{|t e^{i\theta}+s|^2}\right) \\
&\le t\sin\theta
\left(b+
\frac{1}{\cos^2(\theta/2)}\int_{0+}^\infty\frac{s\,\sigma(ds)}{(t+s)^2}\right) 
= t\sin\theta
\left(b+ \frac{\varphi'(t)-b}{\cos^2(\theta/2)}\right)\\
&\le
 \frac{\sin\theta}{\cos^2(\theta/2)}\,t\varphi'(t)
=2\tan(\theta/2)\,t\varphi'(t). \qedhere
\end{align*}
\end{proof}

For complete Bernstein functions we can extend the resolvent formula of Lemma \ref{In2} to cover greater ranges for the parameters.

\begin{lemma}\label{In1}
Let $\varphi\in \mathcal{CBF}$, $q>1$, $\lambda \in \Sigma_{\pi/q}$ and $z \in \Sigma_{\pi - \pi/q}$.
Then 
\begin{multline}\label{R11}
(z+\varphi(\lambda))^{-1} \\
=\frac{1}{z+\varphi(\infty)}
+\frac{q}{\pi}\int_0^\infty
\frac{\operatorname{Im}\varphi(te^{i\pi/q})\,t^{q-1}}
{(z+\varphi(te^{i\pi/q}))(z+\varphi(te^{-i\pi/q}))(\lambda^q+t^q)} \,dt.
\end{multline}
Moreover the integral is absolutely convergent, and, for $\psi \in (0, \pi-\pi/q)$, there are constants $C_1$, $C_2$ (depending on $q$ and $\psi$) such that, for $\lambda \in \Sigma_\theta$ where $\theta \in (0,\pi/q)$, and $z \in \Sigma_\psi$,
\begin{multline*}
\left| \int_0^\infty 
\frac{\operatorname{Im}\varphi(te^{i\pi/q})\,t^{q-1}}
{(z+\varphi(te^{i\pi/q}))(z+\varphi(te^{-i\pi/q}))(\lambda^q+t^q)} \,dt\right|  \\
\le C_1 \int_0^\infty \frac{|\operatorname{Im} \varphi(te^{i\pi/q})|}{(|z|+\varphi(t))^2 t} \, dt
\le \frac{C_2}{|z|}.
\end{multline*}
\end{lemma}

\begin{proof}
The proof of (\ref{R11}) proceeds in the same way as Lemma \ref{In2}, with Proposition \ref{pcbf} replacing Proposition \ref{pnpp}.  In this case the integral can be shown to be absolutely convergent in the same way as for Bernstein functions in Section \ref{pnp}, using (\ref{R2}) instead of (\ref{BFD}) for the second inequality.  
\end{proof}

\begin{example} \label{exlog}
Consider the complete Bernstein function $\varphi(z) = \log(1+z)$, and let $q=2$.  For $z,\lambda\in\C_+$, Lemma \ref{In1} gives
\begin{align*} \label{exlog}
(z + \log(1+\lambda))^{-1} &= \frac{2}{\pi} \int_0^\infty \frac{ t \tan^{-1} t }{ \left(\big(z + \log \sqrt{1+t^2}\big)^2 + (\tan^{-1} t)^2 \right) (\lambda^2 + t^2) } \, dt \\
&= \int_0^{\pi/2} \frac{s \sin s} {\left( (z - \log \cos s)^2 + s^2\right) \cos^3 s (\lambda^2 + \tan^2 s)} \, ds.
\end{align*}
This may be compared with a formula in \cite[Example 2]{Mi98}.
\end{example}

\section{Preservation of sectorial operators} \label{pres}

In this section we give the main results about preserving sectoriality of operators under $\NPp$-functions.  For a closed linear operator $A$ on a complex Banach space $X$ we denote by $\dom(A)$, $\ran(A)$ and $\sigma(A)$ the {\em domain}, the {\em range} and the {\em spectrum} of $A$, respectively.    The space of bounded linear operators
on $X$ is denoted by $\mathcal L(X)$.

Let us recall that an operator $A$ on a Banach space $X$ is \emph{sectorial}
of angle $\omega\in [0,\pi)$ if $A$ is closed and densely defined,
$\sigma(A)\subset\overline{\Sigma}_\omega$
and for every $\omega'\in (\omega,\pi)$
there exists $M(A,\omega')<\infty$ such that
\[
\|z(A+z)^{-1}\|\le M(A,\omega'), \quad z\in \Sigma_{\pi-\omega'}.
\]
The set of sectorial operators of angle $\omega\in [0,\pi)$ on a Banach space $X$ will be denoted by
$\operatorname{Sect}_X(\omega)$, or simply by $\operatorname{Sect}(\omega)$ when no confusion is likely.  It is a standard fact in semigroup theory that $A \in \operatorname{Sect}(\omega)$ for some $\omega \in [0,\pi/2)$ if and only if $-A$ is the generator of a sectorially bounded holomorphic $C_0$-semigroup on $X$ of angle $\pi/2 -\omega$, i.e.\ $-A$ is the generator of a $C_0$-semigroup $(T(t))_{t\ge0}$ on $X$ which has a holomorphic extension to $\Sigma_{\pi/2-\omega}$ satisfying
\[
\sup \{ \|T(z)\| : z \in \Sigma_\theta \} < \infty, \qquad 0<\theta<\pi/2-\omega.
\]

We refer the reader to \cite[Section 2.1]{Ha06} for basic properties of sectorial operators. 
The reader should be aware that some texts, including \cite{Ha06}, do not require sectorial operators to be densely defined, and some texts require other properties.  We shall sometimes make explicit additional assumptions on our sectorial operators $A$, that $A$ is injective or it has dense range.  Note that any sectorial operator with dense range is injective, and we can consider the inverse operator $A^{-1}$ with $\dom(A^{-1}) = \ran(A)$.   Then $A^{-1}$ is densely defined, and it is sectorial with the same angle as $A$.
 
When $X$ is reflexive, the density of $\dom(A)$ can be omitted from the definition above, as it follows from the other properties.  If $A$ is sectorial and injective on a reflexive space, then $A$ has dense range.

When $A$ is a sectorial operator, $A(1+A)^{-2}$ is a bounded sectorial operator, and its range is $\dom(A) \cap \ran(A)$.  Moreover $A(1+A)^{-2}$ has dense range if $A$ has dense range.

We consider $A \in \operatorname{Sect}(\omega)$, where $\omega \in [0,\pi/2)$, and a function $f \in \NPp$, and ask when $f(A) \in \operatorname{Sect}(\omega)$.  It is unrealistic to give a complete charaterisation of all pairs $(A,f)$ for which there is a positive answer, but we can address two subsidiary questions.  One of them is a reformulation of (Q1) from the Introduction when $\omega \in [0,\pi/2)$ and of (Q2) when $\omega \in [0,\pi)$.
\begin{enumerate}
\item[(Q1$'$)]  For which functions $f \in \NPp$ is $f(A) \in \operatorname{Sect}(\omega)$ for all $A \in \operatorname{Sect}(\omega)$ on a given Banach space $X$?
\end{enumerate}
The other question is:
\begin{enumerate}
\item[(Q5)]  For which operators $A \in \operatorname{Sect}(\omega)$ is $f(A) \in \operatorname{Sect}(\omega)$ for all $f \in \NPp$?
\end{enumerate}

First we recall how $f(A)$ is defined, and point out why we sometimes have to assume that $A$ is injective or that $A$ has dense range.  There is no great loss in this, since we can confine our attention to the injective part of $A$ \cite[p.24 and Corollary 2.3.9]{Ha06}.

Let $A \in \operatorname{Sect}(\omega)$ where $\omega < \pi/2$.  For $f \in \NPp$ and $\theta \in (\omega,\pi/2)$, Proposition \ref{pnpp}(\ref{pnpp3}) shows that 
\[
|f(z)| = O(|z|^{-1}), \quad |z|\to0, z\in \Sigma_\theta, \qquad |f(z)| = O(|z|), \quad |z|\to\infty, z\in \Sigma_\theta.
\]
Then $f(A)$ can be defined by the holomorphic functional calculus \cite[Proposition 2.3.13]{Ha06}.  In particular, $f$ is regularised by the function $\tau^2$, where $\tau(z) = z/(1+z)^2$.  This means that $(f\cdot\tau^2)(A)$ is a bounded operator defined on $X$ by a Cauchy integral, and then 
\begin{align*}
\dom(f(A)) &= \left\{ x \in X: (f \cdot \tau^2)(A)x \in \ran(\tau^2(A)) \right\}, \\
 \tau^2(A)f(A)x &= (f \cdot \tau^2)(A)x.
\end{align*}
For this to be a single-valued operator $\tau^2(A)$ must be injective.  In this case $\tau^2(A) = (A (I+A)^{-2})^2$, and so we have to assume that $A$ is injective.  
Then, by \cite[Theorem 1.3.2c)]{Ha06},
\[
(f\cdot\tau^2)(A)x = f(A) \tau^2(A)x, \qquad x \in X.
\]
Thus the domain of $f(A)$ contains the range of $\tau(A)^2$.  If $A$ has dense range, then $f(A)$ has dense domain and $\dom(A) \cap \ran(A)$ is a core for $f(A)$ \cite[Proposition 2.3.13b)]{Ha06}.   Moreover $f(A)$ has dense range, because $\dom\left((1/f)(A)\right)$ has dense domain and
\[
f(A)(1/f)(A)x = x, \qquad x \in \dom\left((1/f)(A)\right),
\]
by \cite[Proposition 1.2.2d)]{Ha06}.

When $A \in \operatorname{Sect}(\omega)$ is injective and $f \in H^\infty(\Sigma_\theta)$ where $\theta \in (\omega,\pi)$, we can define $f(A)$ by the same method as above and the same properties hold.  We shall use standard properties of $H^\infty$-functional calculus on sectors $\Sigma_\theta$ for sectorial operators $A$, as described in \cite[Chapters 2,5]{Ha06}, \cite{KW04}.  In particular we say that $A$ has \emph{bounded $H^\infty$-calculus} on $\Sigma_\theta$ if $f(A) \in \mathcal{L}(X)$ for all $f \in H^\infty(\Sigma_\theta)$.  Note that then $\|f(A)\| \le C\|f\|_{H^\infty(\Sigma_\theta)}$ for some constant $C$ (depending on $A$) \cite[p.112]{Ha06}.  We shall make free use of the Composition Rule \cite[Theorem 2.4.2]{Ha06}, the characterisation of semigroup generators given in \cite[Proposition 2.5]{BHM13}, and the compatibility of the half-plane calculus and the sectorial calculus \cite[Proposition 2.8]{BHM13}.  A special case of the Composition Rule shows that if $A$ has bounded $H^\infty$-calculus on $\Sigma_\theta$ and $\alpha \in (0,\pi/\theta)$, then $A^\alpha$ has bounded $H^\infty$-calculus on $\Sigma_{\alpha\theta}$.  


When $-A$ generates a bounded $C_0$-semigroup $(T(t))_{t\ge0}$ and $f$ is a Bernstein function, we can define $f(A)$ as in \cite[Chapter 13]{SchilSonVon2010} (the reader should be aware that in \cite{SchilSonVon2010}, the notation for $f(A)$ is $A^f$ and $A$ denotes the generator of the $C_0$-semigroup, so there are some differences of sign). Then $\dom(A)$ is a core for $f(A)$ and 
\begin{equation} \label{bpfc}
f(A)x = ax + bAx + \int_{0+}^\infty (x - T(t)x) \, \mu(dt), \qquad x \in \dom(A),
\end{equation}
where $(a,b,\mu)$ is the L\'evy triple for $f$ \cite[Theorem 13.6]{SchilSonVon2010}.  When $A$ is injective, this definition of $f(A)$ agrees with the definition by the holomorphic functional calculus \cite[Proposition 3.6]{GT}.

  The fractional powers $A^q$ of a sectorial operator $A$ can be defined by the extended functional calculus for $q>0$, and for $q \in \R$ if $A$ is injective.  Note that $z^q \in \NPp$ if and only if $|q|\le1$, but the higher fractional powers can be regularised by $\tau^n$ for large $n \in \N$.   It is well known that  fractional powers preserve sectoriality of an operator in the following sense (see \cite[Corollary 3.10]{BBL} or \cite[Proposition 3.1.2]{Ha06}).

\begin{prop}\label{PrBBL}
Let $A\in \operatorname{Sect}(\omega)$, $\omega\in [0,\pi)$,
and let $q>0$ be such that $q\omega<\pi$.
Then $A^q\in \operatorname{Sect}(q\omega)$.
\end{prop}

In the Appendix (Section \ref{frac}) we shall give explicit estimates for the sectoriality constants of $A^q$ in terms of those of $A$.

Now we give an example to show that there are functions $f \in \mathcal{D}$ such that $f(A)$ cannot be defined by the extended Hille-Phillips calculus as described in \cite[Section 3.3]{GT}.   For $f(A)$ to be defined in that calculus for some sectorial operator $A$, there must exist a regulariser $e$ for $f$ in the sense of the Hille-Phillips calculus.  In particular, this requires that $e \in H^\infty(\C_+) \cap C(\overline{\C}_+)$, $e\ne0$ and $e\cdot f \in H^\infty(\C_+)$.

\begin{example} \label{nonhp}
Let $(r_k)_{k\ge1}$ be an enumeration of the rational numbers in $[1,2]$, let $a_1=1$ and 
\[
a_k = 2^{-k} \min_{j=1,\dots,k-1} |r_j-r_k|, \qquad  k=2,3,\dots.
\]
Define 
\[
f(z) = \sum_{k=1}^\infty a_k \frac{z}{z^2 + r_k^2}, \qquad z \in \C_+.
\]
By Example \ref{compsupp}, $f \in \mathcal{D}$.  Suppose that $e \in H^\infty(\C_+) \cap C(\overline{\C}_+)$ and $e \cdot f \in H^\infty(\C_+)$.  Take $j\ge1$.  For $z = s+ir_j$ where $s \in (0,1)$, we have
\[
\left|\sum_{k=j+1}^\infty a_k \frac{z}{z^2 + r_k^2} \right| \le \sum_{k=j+1}^\infty   \frac{|r_j-r_k|\sqrt5} {2^{k}|(z-ir_k)(z+ir_k)|}  \le \sum_{k=j+1}^\infty \frac{\sqrt5}{2^{k+1}} < 2^{-j+1}.
\]
Hence
\[
|f(s+ir_j)| \ge \left|\frac {a_j(s+ir_j)}{(s+ir_j)^2 + r_j^2}\right| - \left|\sum_{k=1}^{j-1} \frac{a_k(s+ir_j)}{(s+ir_j)^2 + r_k^2} \right| - 2^{-j+1} \to \infty
\]
as $s \to 0+$.  Since $e \cdot f$ is bounded on $\C_+$ and $e$ is continuous on $\overline{\C}_+$, it follows that $e(ir_j) = 0$ for each $j\ge 1$ and then $e(is)=0$ for all $s \in [1,2]$.  Since $e \in H^\infty(\C_+)$, it follows that $e=0$.  Thus $f$ is not regularisable in the Hille-Phillips functional calculus for any sectorial operator $A$.
\end{example}

Our first result shows that the answer to the question (Q5) is very closely related to boundedness of $H^\infty$-calculus on a sector.  It is well known that there are examples of sectorial operators (even of angle $0$ and on Hilbert space), which do not have bounded $H^\infty$-calculus on any sector (see \cite[Theorem 9.1.7, Corollary 9.1.8]{Ha06}, \cite[Theorem 3.6]{Fa15}, and Proposition \ref{limits} and Remark \ref{opsrem} below).  The following result shows that, for any such operator $A$, there are functions $f\in\NPp$ such that $-f(A)$ does not generate a $C_0$-semigroup.

\begin{thm} \label{ops}
Let $A$ be a sectorial operator of angle $\omega \in [0, \pi/2)$ with dense range, and let $\theta \in (\omega,\pi/2)$.  Consider the following properties:
\begin{enumerate}[{\rm(i)}]
\item \label{fci} $A$ has bounded $H^\infty$-calculus on $\Sigma_\theta$.
\item \label{fcii} For every $f \in \NPp$, $f(A)$ is a sectorial operator of angle (at most) $\theta$.
\item \label{fciii} For every $f \in \NPp$, $-f(A)$ is the generator of a bounded $C_0$-semigroup.
\item \label{fciv} For every $f \in \NPp$ such that $f(0+)$ and $f(\infty)$ both exist in $(0,\infty)$, $-f(A)$ is the generator of a $C_0$-semigroup.
\item \label{fcv} $A$ has bounded $H^\infty$-calculus on $\C_+$.
\end{enumerate}
Then
\[
\mbox{\rm(\ref{fci})}\implies\mbox{\rm(\ref{fcii})}\implies\mbox{\rm(\ref{fciii})}\iff\mbox{\rm(\ref{fciv})}\iff\mbox{\rm(\ref{fcv})}.
\]
If $X$ is a Hilbert space then all these properties are equivalent.
\end{thm}

\begin{proof}  It is trivial that (ii) implies (iii) and (iii) implies (iv).  

Suppose for a contradiction that (iv) is true and (v) is false.  By \cite[Proposition 5.3.4]{Ha06} the functional calculus is unbounded on the subalgebra $\mathcal{A} := H^\infty(\C_+) \cap C_0(\overline{\C}_+)$ so there exists $h \in \mathcal{A}$ such that $\|h\|_{H^\infty(\C_+)} = 1$ and $h(A)$ is unbounded.  Let
\begin{equation}\label{1}
h_0(z):=\frac{h(z)+\overline{h(\overline{z})}}{2} \quad \text{and} \quad h_1(z):=\frac{h(z)-\overline{h(\overline{z})}}{2i}
\end{equation}
Then
\[
h=h_0 + ih_1,
\]
and
\[
h_i \in \mathcal{A}, \quad \|h_i\|_\infty\le1, \quad h_i((0,\infty))\subset \mathbb R, \quad h_i(z)=\overline{h_i(\overline{z})}, \qquad i=0,1.
\]
At least one of $h_0(A)$ and $h_1(A)$ is unbounded.  We choose $i$ so that $h_i(A)$ is unbounded, and we define
\[
g(z):= \frac{h_i(z)+2}{4}.
\]
Then $g\in \NPp$ and
\[
\frac{1}{4}\le |g(z)|\le \frac{3}{4}, \qquad z \in \mathbb C_+.
\]
Taking the branch of the logarithm which is real on $(0,\infty)$, we have
\[
f:=-\log g\in \NPp,
\]
and $f(0+)$ and $f(\infty)$ both exist in $[\log(4/3),\log4]$.  Moreover,
\[
g(z)=\exp{\left(-f(z)\right)}.
\]
Since $f$ maps $\Sigma_\theta$ to $\Sigma_\theta$ for $\theta \in\ (\omega,\pi/2)$, and  $f(A) \in \operatorname{Sect}(\pi/2)$ by assumption, the Composition Rule implies that
\[
\exp({-f(A)})=g(A)
\]
which is an unbounded operator.  Applying \cite[Proposition 2.5]{BHM13} shows that 
$-f(A)$ does not generate a $C_0$-semigroup on $X$ which contradicts our assumption.  This proves that (iv) implies (v).

Now suppose that (v) is true, and let $f\in\NPp$.   For $t\ge0$, the function $g_t(z) = \exp(-tf(z))$ belongs to $H^\infty(\C_+)$ with $\|g_t\|_{H^\infty(\C_+)} \le 1$.  By the Composition Rule, $g_t(A) = \exp(-tf(A))$.  By (v), $\exp(-tf(A)) \in \mathcal{L}(X)$ and $\|\exp(-tf(A))\| \le C$.  By \cite[Proposition 2.5]{BHM13}, $f(A)$ generates a bounded $C_0$-semigroup.   Thus (iii) is true.

Next, suppose that (i) holds, and let $\alpha = 2\theta/\pi < 1$.  By the Composition Rule, $A^{1/\alpha}$ has bounded $H^\infty$-calculus on $\C_+$.  Given $f \in \NPp$, define
$$
g(z) = f(z^\alpha)^{1/\alpha}, \quad z \in\C_+.
$$
Then $g \in \NPp$.  Applying the implication (\ref{fcv})$\implies$(\ref{fciii}) to $A^{1/\alpha}$ shows that $-g(A^{1/\alpha})$ generates a bounded $C_0$-semigroup, and is therefore sectorial of angle $\pi/2$.  Then
\[
f(A) = g(A^{1/\alpha})^\alpha
\]
is sectorial of angle $\alpha\pi/2 = \theta$.

When $X$ is a Hilbert space, the implication (\ref{fcv})$\implies$(\ref{fci}) holds by McIntosh's theorem \cite[Theorem 7.3.1]{Ha06}.
\end{proof}

For general Banach spaces, it is not true that (\ref{fcv}) implies (\ref{fci}) in Theorem \ref{ops}, due to a counterexample given in \cite{Kal}.   It is an open question whether (\ref{fcv}) implies (\ref{fcii}).

Now we turn to the question (Q1).  For a given Banach space $X$, let $\mathcal{F}_X$ be the set of all  functions $f \in \NPp$ with the property that $f(A) \in \operatorname{Sect}_X(\omega)$ whenever $\omega \in [0,\pi/2)$, $A \in \operatorname{Sect}_X(\omega)$ and $A$ has dense range in $X$.   First we show that $\mathcal{F}_X$ possesses some algebraic structure.

\begin{prop} \label{fsubx}
Let $X$ be a Banach space, and $f,g \in \mathcal{F}_X$.  Then the following functions are in $\mathcal{F}_X$:
\[
1/f, \qquad g \circ f,  \qquad f+g.
\]
\end{prop}

\begin{proof}
Let $A \in \operatorname{Sect}_X(\omega)$ with dense range, where $\omega \in (0,\pi/2)$.  Then $f(A) \in \operatorname{Sect}_X(\omega)$ with dense range, and then $g(f(A)) \in \operatorname{Sect}_X(\omega)$.  Now $(g \circ f)(A) = g(f(A))$, by the Composition Rule.  Thus $g \circ f \in \mathcal{F}_X$.  Since $1/z \in \mathcal{F}_X$, it follows that $1/f \in \mathcal{F}_X$.

Furthermore, $-f(A)$ and $-g(A)$ are the generators of bounded holomorphic $C_0$-semigroups $(S(t))_{t\ge0}$ and $(T(t))_{t\ge0}$, respectively, of angle $\pi/2-\omega$.  These semigroups commute because the resolvents $(s + f)(A)^{-1}$ and $(u+g)(A)^{-1}$ of their generators commute for $s,u>0$ (see \cite[Theorem 1.3.2a)]{Ha06}).  Hence $\left(S(t)T(t)\right)_{t\ge0}$ is a bounded holomorphic $C_0$-semigroup of angle $\pi/2-\omega$ and its generator $B$ is a closed extension of $f(A) + g(A)$  whose domain $\dom(f(A)) \cap \dom (g(A))$ is dense and invariant under the semigroup.  Then $B$ is the closure of $f(A) + g(A)$.  By \cite[Proposition 1.3.2c)]{Ha06}, $(f+g)(A)$ is a closed extension of $f(A) + g(A)$ and hence an extension of $B$.  Since $I +B$ is surjective and $I + (f+g)(A)$ is injective by \cite[Proposition 1.3.2f)]{Ha06}, it follows that $B = (f+g)(A)$.  Thus $f+g \in \mathcal{F}_X$.
\end{proof}

\begin{remark}
\begin{enumerate}[\rm1.]
\item  In the context of Proposition \ref{fsubx} we do not know whether $(fg)^{1/2}$ is necessarily in $\mathcal{F}_X$, or equivalently whether $(fg)(A) \in \operatorname{Sect}_X(2\omega)$.  See Remark \ref{rempr}.
\item  Instead of $\mathcal{F}_X$, one might consider the class of $\NPp$-functions preserving sectoriality for all injective operators $A \in \operatorname{Sect}(\omega)$ on $X$ (without requiring dense range).  This class is closed under compositions and sums, but it does not include $1/z$.
\item  One might consider the set $\mathcal{F}$ of all functions $f$ which belong to $\mathcal{F}_X$ for all Banach spaces $X$.  Note that $\mathcal{F}$ is a set because it suffices to consider separable Banach spaces, and therefore it suffices to consider closed subspaces of $\ell^\infty$.
\end{enumerate}
\end{remark}

Next we show that, if $X$ has a conditional basis and $f \in \mathcal{F}_X$, then the limiting values $f(\infty)$ and $f(0+)$ must exist in $[0,\infty]$.

\begin{prop} \label{limits}
Let $f \in \NPp$ be a function such that $f(\infty)$ does not exist in $[0,\infty]$, and let $X$ be a Banach space with a conditional basis $(e_n)_{n\ge1}$.  Then there exists an operator $A \in \operatorname{Sect}_X(0)$, with dense range, such that $-f(A)$ does not generate a $C_0$-semigroup.
\end{prop} 

\begin{proof}
Since the basis is conditional, there exists a sequence $(b_n)_{n\ge 1}$, where $b_n=\pm 1$, such that the multiplication operator
$B$ with 
\[
Be_n= b_n, \qquad n\ge1,
\]
is not bounded.  The operator $B$ is then given by
\begin{align*}
\operatorname{dom}(B) &= \Big\{ {\textstyle{ x = \sum_{n=1}^\infty x_n e_n : \text{$\sum_{n=1}^\infty b_n x_n e_n$ converges in $X$}}} \Big\},  \\
 Bx &= {\textstyle{\sum_{n=1}^\infty b_n x_n e_n}}.
\end{align*}

Let $g(z) = \exp(-f(z))$, and
\[
a_1:=\underline{\lim}_{t\to \infty}\,g(t), \qquad
a_2:=\overline{\lim}_{t\to \infty}\,g(t).
\]
so $0 \le a_1 < a_2 \le 1$ by assumption.  Choose $c_1,c_2$ such that $a_1 < 2c_1 < 2c_2 < a_2$.
By the intermediate value theorem there is a strictly increasing sequence $(t_n)$ in $[1,\infty)$, such that 
 $t_n\to\infty$ as $n\to\infty$ and
\[
g(t_n)= \begin{cases} 2c_1 &\text{if $b_n = -1$}, \\ 2c_2 &\text{if $b_n = 1$}. \end{cases}
\]
Let $A$ be the multiplication operator on X with
\[
Ae_n=t_ne_n, \qquad n \ge 1.
\]
Then $A \in \operatorname{Sect}_X(0)$ \cite[Lemma 9.1.2]{Ha06}, and it is clear that $A$ has dense range.  Moreover $g(A)$ is the multiplication operator with
\[
g(A)e_n = g(t_n)e_n = \left( (c_1+c_2) + (c_2-c_1)b_n\right) e_n,
\]
so $g(A) = (c_1+c_2) + (c_2-c_1)B$, which is unbounded.  As in the proof of Theorem \ref{ops}, it follows that $\exp(-f(A))$ is unbounded, and $-f(A)$ does not generate a $C_0$-semigroup.
\end{proof}

\begin{remark} \label{opsrem}
The proof of Proposition \ref{limits} shows that the operator $A$ constructed therein does not have bounded $H^\infty$-calculus on $\C_+$, and the argument can easily be modified to show that $A$ does not have bounded $H^\infty$-calculus on $\Sigma_\theta$ for any $\theta \in (0,\pi)$. The fact that a multiplication operator with respect to a conditional basis may not have bounded $H^\infty$-calculus is well known, going back at least to \cite{BC}, but Carleson's interpolation theorem is usually invoked to construct the function $g$ such that $g(A)$ is unbounded.
\end{remark}

\begin{cor} \label{limit0}
Let $f \in \NPp$ be a function such that $f(0+)$ does not exist in $[0,\infty]$, and let $X$ be a Banach space with a conditional basis $(e_n)_{n\ge1}$.  Then there exists an operator $A_0 \in \operatorname{Sect}_X(0)$ such that $-f(A_0)$ does not generate a $C_0$-semigroup.
\end{cor}

\begin{proof}  One may apply Proposition \ref{limits} to $f(1/z)$ and obtain that $-f(A^{-1})$ does not generate a $C_0$-semigroup.
\end{proof}

In the remainder of this section, we give positive results showing that certain functions $f$  have the property that $f(A)$ is sectorial either for all sectorial $A$, or for all injective sectorial $A$, with sectorial angle either in $[0,\pi/2)$ or in $[0,\pi)$.  First we obtain an integral representation for the resolvent of $f(A)$ matching the scalar versions in Lemma \ref{In2}.  The following theorem describes two cases where this can be achieved.  When $A$ is injective, the representation holds for all functions $f \in \NPp$ such that $J_\theta(r;f)$ is finite for all $\theta\in (0,\pi/2)$ and some (or all) $r>0$.   However, we restrict attention to the class $\mathcal{E}$ which is the appropriate class for the later corollaries.

\begin{thm}\label{TCbf11}
Let $A\in {\rm Sect}(\omega)$ for some $\omega\in [0,\pi/2)$,  let $q\in \left(2,\pi/\omega\right)$ and $\psi =\pi\left(1-\frac{1}{q}\right)$.  Assume that either
\begin{enumerate} [\rm \;a)]
\item \label{TCbfii} $A$ is injective and $f \in \E$, or
\item \label{TCbfi} $f \in \mathcal{BF}$.
\end{enumerate}
Let $z \in \Sigma_\psi$.  Then $z+f(A)$ is invertible and
\begin{align}  \label{Ra+}
 \lefteqn{(z + f(A))^{-1}} \\
 &= (z+f(\infty))^{-1}I  +\frac{q}{\pi}\int_0^\infty 
\frac{\operatorname{Im} f(te^{i\pi/q})\,t^{q-1}}
{(z+f(te^{i\pi/q}))(z+f(te^{-i\pi/q}))} \,(A^q+t^q)^{-1}\,dt \nonumber\\
 &= (z+f(0+))^{-1}I  -\frac{q}{\pi}\int_0^\infty
\frac{\operatorname{Im} f(te^{i\pi/q})}
{(z+f(te^{i\pi/q}))(z+f(te^{-i\pi/q}))t} \,A^q (A^q+t^q)^{-1}\,dt.  \nonumber
\end{align}
Here the integrals are absolutely convergent.
\end{thm}

\begin{proof}  We consider only the representation in the second line of (\ref{Ra+}) which we denote by $R_q(z; f,A)$. The other representation is similar.   Let $\theta \in (0,\psi)$ and $f \in \E$.  Proposition \ref{PrBBL} shows that $\|(A^q+t^q)^{-1}\| \le C_qt^{-q}$ for some constant $C_q$. Then the estimate (\ref{integ}) with $\lambda$ replaced by $A$, and Definition \ref{defcong}, imply that the integral in $R_q(z; f,A)$ converges uniformly for $z \in \Sigma_\theta$.  Since the integrand is holomorphic in $z$, it follows that $R_q(\cdot; f,A)$ is a holomorphic function from $\Sigma_\psi$ to $\mathcal{L}(X)$. By analytic continuation (see \cite[Proposition B.5]{ABHN01}), it suffices to prove (\ref{Ra+}) only for $z=s>0$.  

Suppose first that $A$ is injective and $f \in \E$. Then  $\tau(\lambda) = \lambda/(1+\lambda)^2$ is a regulariser for $(s+f(\lambda))^{-1}$ for any $s>0$. Thus, 
using the holomorphic functional calculus and taking $\omega'\in (\omega,\pi/2)$, we can write 
\begin{align}\label{SSR+}
(s+f(\lambda))^{-1}(A) &=[A(1+A)^{-2}]^{-1}
\left(\frac{\lambda}{(\lambda+1)^2(s+f(\lambda))}\right)(A) \\
&=\frac{1}{2\pi i} [A(1+A)^{-2}]^{-1} \int_{\partial \Sigma_{\omega'}} \frac{\lambda}{(\lambda+1)^2}\frac{(A-\lambda)^{-1}}{(s+f(\lambda))}\,d\lambda,  \nonumber\\
A(A+1)^{-2}(A^q+t^q)^{-k} &= \frac{1}{2\pi i}\int_{\partial \Sigma_{\omega'}}
\frac{\lambda (A-\lambda)^{-1}}{(\lambda+1)^2}\frac{d\lambda}
{(\lambda^q+t^q)^k}, \qquad t>0,\; k=0,1. \nonumber
\end{align}
Since $f\in\E$, we can use Lemma \ref{In2} and Fubini's theorem to obtain
\begin{eqnarray*}
\lefteqn{\left(\frac{\lambda}{(\lambda+1)^2(s+f(\lambda))}\right)(A)} \\
 &=& \frac{1}{2\pi i(s+f(\infty))}\int_{\partial \Sigma_{\omega'}}\frac{\lambda}{(\lambda+1)^2}(A-\lambda)^{-1}\,d\lambda  \\
&&+\frac{q}{2\pi^2 i}\int_0^\infty
\frac{\operatorname{Im}f(te^{i\pi/q})\,t^{q-1}}
{(s+f(te^{i\pi/q}))(s+f(te^{-i\pi/q}))}
\int_{\partial \Sigma_{\omega'}}
\frac{\lambda (A-\lambda)^{-1}}{(\lambda+1)^2}\frac{d\lambda}
{(\lambda^q+t^q)}\,dt  \\
&=& \frac{A(A+1)^{-2}}{s+f(\infty)}
+\frac{q}{\pi}A(A+1)^{-2}\int_0^\infty
\frac{\operatorname{Im}f(te^{i\pi/q})\,t^{q-1}}
{(s+f(te^{i\pi/q}))(s+f(te^{-i\pi/q}))}
(A^q+t^q)^{-1}\,dt \\
&=& A(A+1)^{-2}R_q(s;f,A).
\end{eqnarray*}
Therefore by (\ref{SSR+}), 
\[
(s+f(\lambda))^{-1}(A)=R_q(s;f,A),\quad s>0.
\]
Then, by the product rule for the holomorphic functional calculus \cite[Theorem 1.3.2c)]{Ha06},
\[
R_q(s;f, A)(f(A)+s)\subset I.
\]
This means that $R_q(s;f,A)=(f(A)+s)^{-1}$,
so  (\ref{Ra+}) holds for $s>0$, and then for all 
$z\in \Sigma_\beta$.  This completes the proof for case \ref{TCbfii}), and for case \ref{TCbfi}) when $A$ is injective.

Now we consider the case when $A$ is not injective, and $f \in \mathcal{BF}$.  Let
\[
A_\varepsilon:=A+\varepsilon,\quad \varepsilon\in (0,1].
\]
The operators $A_\varepsilon, \;\varepsilon\in (0,1]$, are uniformly sectorial in the sense that
\[
\sup  \left\{\|z(A_\varepsilon+z)^{-1}\| : \varepsilon \in (0,1], z \in\Sigma_{\pi-\omega'} \right\} < \infty, \qquad \omega' \in (\omega,\pi)
\]
(see (\ref{MEs0})).  By the case above, we have
\begin{equation}\label{Reps+}
(z+f(A_\varepsilon))^{-1}=R_q(z;f,A_\varepsilon),\quad z\in \Sigma_\psi.
\end{equation}
By \cite[Proposition 3.1.9]{Ha06}, we also have
\[
\lim_{\varepsilon\to 0+}\,[A^q-(A_\varepsilon)^q]x=0,\quad x\in \dom(A^q)=\dom((A_\varepsilon)^q).
\]
Then, using
\[
(A_\varepsilon^q+t^q)^{-1}-(A^q+t^q)^{-1}
=(A_\varepsilon^q+t^q)^{-1}
[A^q-(A_\varepsilon)^q](A^q+t^q)^{-1},
\]
we infer that
\[
\lim_{\varepsilon\to 0+}\,\|(A_\varepsilon^q+t^q)^{-1}x-(A^q+t^q)^{-1}x\|=0, \qquad x \in X.
\]
Thus, letting $\varepsilon \to 0+$ in (\ref{Reps+}) and using the bounded convergence theorem,
we get that
\[
\lim_{\varepsilon\to 0+} R_q(z;f,A_\varepsilon)x = R_q(z;f,A)x.
\]

Now, let $x \in \dom(A)$.  Since $R_q(z;f,A_\varepsilon)$ commutes with $(1+A)^{-1}$, $R_q(z;f,A_\varepsilon)x \in \dom(A)$.  Moreover, using (\ref{bpfc}) and the compatibility of calculi,
\begin{align*}
\lefteqn{f(A) R_q(z;f,A_\varepsilon) x - x} \qquad \\
 &= f(A) R_q(z;f,A_\varepsilon) x - f(A_\varepsilon) R_q(z;f,A_\varepsilon) x \\
&=  \int_{0+}^\infty (1 - e^{-\varepsilon t}) T(t) R_q(z;f,A_\varepsilon) x \,\mu(dt) - \varepsilon b R_q(z;f,A_\varepsilon) x,
\end{align*}
where $(a,b,\mu)$ is the L\'evy triple of $f$ and $(T(t))_{t\ge0}$ is the bounded $C_0$-semigroup generated by $A$.  By the uniform sectoriality of $(A_\varepsilon)_{\varepsilon\in(0,1]}$ there exists $C_z$ such that $\|R_q(z;f,A_\varepsilon)\| \le C_z$.  Since $\|T(t)\| \le K$ for $t\ge0$, we have
\[
\left\|f(A) R_q(z;f,A_\varepsilon) x - x\right\| \le \varepsilon b C_z \|x\| + C_zK \int_0^\infty (1 - e^{-\varepsilon t}) \, \mu(dt) \|x\| \to 0
\]
as $\varepsilon \to 0+$.  Since $f(A)$ is closed, we have that $(z+f(A)) R_q(z;f,A)x = x$ for $x \in \dom(A)$.  Similarly, $ R_q(z;f,A)(z+f(A))x = x$ for $x \in \dom(A)$.  Using that $\dom(A)$ is a dense subspace of $X$ and a core for $f(A)$,   $R_q(z;f,A)$ is bounded and $f(A)$ is closed, it follows that 
$(z+f(A)) R_q(z;f,A)x = x$ for all $x \in X$ and that  $R_q(z;f,A)(z+f(A))x = x$ for all $x \in \dom(f(A))$.
\end{proof}

The following corollary of Theorem \ref{TCbf11} provides a sectorial resolvent estimate for $f(A)$ with explicit constants.  The constants depend on $f$ only through the constant $\kappa_\theta(f)$ in Definition \ref{defcong}, and they are independent of $f$ when $f \in \mathcal{BF}$.  This knowledge may be of value for various approximation procedures when uniformity of a limit is required.

\begin{cor}\label{TCbf12}
Let $A\in {\rm Sect}(\omega)$ for some $\omega\in [0,\pi/2)$, and let $f\in \E$. Assume that either $A$ has dense range or $f \in \mathcal{BF}$.  Then $f(A)\in {\rm Sect}(\omega)$.  

More precisely, when $\theta \in (0,\pi-\omega)$ and 
\[
\max \left( \frac{\pi}{\pi-\theta}, 2 \right) < q < \frac{\pi}{\omega},
\]
one has
\begin{equation}
\label{Ra110}
\left\|z(z+f(A))^{-1}\right\|  
\le \frac{1}{\sin(\pi/q)} +\frac{q\tilde{M}_{\omega,q,0}(A)} {\pi \cos^2(\pi/q)} \frac{\kappa_{\pi/q}(f)}{\cos^2\left((\pi/q+\theta)/2\right)},
\quad  z\in \Sigma_\theta,
\end{equation}
where $\kappa_{\pi/q}(f)$ satisfies {\rm(\ref{tilde})} and  $\tilde{M}_{\omega,q,0}(A)$ is defined by {\rm(\ref{Mn})} in the Appendix.  If $f \in \mathcal{BF}$ then we may take $\kappa_{\pi/q}(f)=\tan(\pi/q)$.
\end{cor}

\begin{proof}  The assumptions imply that $f(A)$ has dense domain.  Let $z\in \Sigma_\theta$.
Using Theorem \ref{TCbf11} and Proposition \ref{PrBBL10}, and estimating similarly to (\ref{integ}), we have
\begin{align*}
\lefteqn{\left\|z(z+f(A))^{-1}\right\|} \quad\\
&\le \frac{|z|}{|z+f(\infty)|}+\frac{q|z|}{\pi} \int_0^\infty
\frac{|\operatorname{Im}f(te^{i\pi/q})| \,t^{q-1}}
{|z+f(te^{i\pi/q})| \, |z+f(te^{-i\pi/q})|} \, \frac{\tilde{M}_{\omega,q,0}(A)}{t^q} \, dt\\
&\le \frac{|z|}{|z+f(\infty)|} + \frac{q\tilde{M}_{\omega,q,0}(A) |z| J_q(|z|;f)} {\pi \cos^2(\pi/q)\cos^2\left((\pi/q+\theta)/2\right)},
\end{align*}
where ${J}_q(|z|;f)$ is defined by  (\ref{tilde}).  Since $f(\infty) \ge 0$, we have $|z+f(\infty)| \ge |z|$ if $z \in \C_+$.  If $\theta > \pi/2$ and $z \in \Sigma_\theta \setminus \C_+$, then
\[
|z + f(\infty)| \ge  |z| \sin(\pi-\theta) \ge |z| \sin (\pi/q),
\]
since $\pi-\theta > \pi/q$.  Since $|z| J_q(|z|;f) \le \kappa_{\pi/q}(f)$, we obtain (\ref{Ra110}).   When $f \in \mathcal{BF}$, (\ref{BFD}) shows that we may apply (\ref{Ctheta}) in the proof of Proposition \ref{lemDE} with $\theta = \pi/q$,  $a=\infty$, $b=\cos(\pi/q)$ and $c = \sin(\pi/q)$ to show that we may take $\kappa_{\pi/q}(f) = \tan(\pi/q)$.
\end{proof}

\begin{cor} \label{TCbf13}
Let $-A$ be the generator of a sectorially bounded holomorphic $C_0$-semigroup of angle $\theta \in (0,\pi/2]$, and assume that $A$ has dense range.  For any $f \in \E$, $-f(A)$ is the generator of a sectorially bounded holomorphic $C_0$-semigroup of angle $\theta$.
%
\end{cor}

\begin{proof}
This is immediate from Corollary \ref{TCbf12}. 
\end{proof}

Recall from Propositions \ref{sphere} and \ref{lemDE} that $\mathcal{E}$ contains both $\mathcal{S}$ and $\mathcal{D}$.  Hence, Corollaries \ref{TCbf12} and \ref{TCbf13} hold for all functions in $\mathcal{S}$ or $\mathcal{D}$, and in particular for the subclasses of $\mathcal{D}$ listed in (D\ref{dcl1})--(D\ref{dcl4}) on p.\pageref{dcl}, when $A$ has dense range.  Thus Corollary \ref{TCbf13} extends \cite[Corollary 4.10]{GT}
 where $f$ was of the form $h + 1/g$ for some $h,g \in \mathcal{BF}$.  Note that \cite[Corollary 4.7]{GT}, showing that Bernstein functions preserve generators of bounded $C_0$-semigroups which have holomorphic extensions to $\Sigma_\theta$, cannot be extended in this way because it fails for $f(z) = 1/z$ \cite[p.183]{GT}.

In \cite[Theorem 1.3]{GT} it is shown that if $g \in \mathcal{BF}$ and $g$ has a continuous extension which maps  $\overline{\Sigma}_\theta$ to $\overline{\C}_+$ and is holomorphic on $\Sigma_\theta$ where $\theta \in (\pi/2,\pi)$, then $g$ has the following {\it improving property}:  If $-A$ is the generator of a bounded $C_0$-semigroup, then the bounded $C_0$-semigroup generated by $-g(A)$ is holomorphic in $\Sigma_{\theta'}$, where $\theta' = \frac{\pi}{2}(1 - \alpha); \; \alpha = \pi/(2\theta) \in (1/2,1)$.  Let $f(z) = g(z^{1/\alpha}), \; z \in \C_+$.  Then $f \in \NPp$ and $g(z) = f(z^\alpha), \; z \in \Sigma_\theta$.  Thus the following result is a variant of \cite[Theorem 1.3]{GT}, addressing (Q4) from the Introduction.

\begin{cor} \label{improve}
Let $f \in \E$, $\alpha \in (1/2,1)$,  $\theta := \pi/(2\alpha)$ and $g(z) = f(z^\alpha)$ for $z \in \Sigma_\theta$.  Let $A \in \operatorname{Sect}(\omega)$ for some $\omega \in [0,\theta)$, and assume that either $A$ has dense range or $f\in\mathcal{BF}$.  Then $g(A)$ is defined in the holomorphic functional calculus and $g(A) \in \operatorname{Sect}(\alpha\omega)$.  In particular, $-g(A)$ is the generator of a bounded holomorphic $C_0$-semigroup of angle $\pi/2-\alpha\omega$.
\end{cor}

\begin{proof}  Since $f \in \NPp$, $g$ has at most polynomial growth at $0$ and $\infty$, and hence $g(A)$ is defined by the holomorphic functional calculus.  By Proposition \ref{PrBBL10}, $A^\alpha \in \operatorname{Sect}(\alpha\omega)$. By Corollary \ref{TCbf12}, $f(A^\alpha) \in \operatorname{Sect}(\alpha\omega)$. By the Composition Rule, $g(A) = f(A^\alpha)$.    
\end{proof} 

When $-A$ generates a bounded $C_0$-semigroup, Corollary \ref{improve} can be applied with $\omega=\pi/2$ to provide an improving property for the function $g$.  When $f \in \mathcal{BF}$,  $g$ is the composition of two Bernstein functions, so $g \in \mathcal{BF}$ and the conclusion of the Corollary is covered by \cite[Theorem 1.3]{GT}.  However, if $f \notin \mathcal{BF}$ and $\alpha \in (1/2,1)$ is sufficiently close to $1$, then $g \notin \mathcal{BF}$, since $\mathcal{BF}$ is closed under pointwise limits \cite[Corollary 3.8(ii)]{SchilSonVon2010}.   Then the conclusion is outside the scope of \cite[Theorem 1.3]{GT}.

\begin{remark} \label{rempr}
 Let $A$ be sectorial of angle $\omega \in (0,\pi/2)$ with dense range, and let $\tilde F_n$ be as in Corollary \ref{Pp2}, so $\tilde F_n \in \mathcal{D} \subset \NPp$.   By Theorem \ref{TCbf11},  $\tilde F_n(A) \in \operatorname{Sect}(\omega)$.  Moreover, Proposition \ref{pnpp}(\ref{pnpp3})  shows that for any $\omega' \in (\omega,\pi/2)$ and for $m=1,2,\dots,n-1$, there exist constants $c_1,c_2>0$ such that
\[
\frac{c_1}{|z|} \le  |\tilde F_m(z)| \le c_2 |z|, \qquad z \in \Sigma_{\omega'}.
\]
Now \cite[Corollary 3.2]{BGT15}, together with the Composition Rule, can be used to show by induction that
\[
\tilde F_n(A)  =  f_1(A^{\alpha_1})^{\beta_1} \dots f_n(A^{\alpha_n})^{\beta_n},
\]
with its natural domain. 

For example, let $f \in \mathcal{BF}$ and $g(z) = z^{1/2}f(z^{1/2})$.  Then
\[
g(A) = A^{1/2}f(A^{1/2})  = f(A^{1/2})A^{1/2} \in \operatorname{Sect}(\theta)
\]
with
\begin{align*}
\dom(g(A)) &= \left\{ x \in \dom \big(f(A^{1/2})\big):  f(A^{1/2})x \in \dom(A^{1/2}) \right\} \\
&=\left\{ x \in \dom(A^{1/2}):  f(A^{1/2})x \in \dom \big(f(A^{1/2})\big) \right\}.
\end{align*}
\end{remark}

Now we state results for complete Bernstein functions applied to sectorial operators of any angle.  The results and their proofs are similar to Theorem \ref{TCbf11} and Corollary \ref{TCbf12},  with Lemmas \ref{phi} and \ref{In1} replacing Lemma \ref{In2}, Proposition \ref{cong}\ref{congb}) and (\ref{integ}).

\begin{thm}\label{TCbf}
Let $A\in {\rm Sect}(\omega)$, $\omega\in [0,\pi)$, $q\in \left(1,\pi/\omega\right)$, $\psi=\pi\left(1-\frac{1}{q}\right)$, 
and let $\varphi\in \mathcal{CBF}$, $z\in \Sigma_\psi$. Then $z+\varphi(A)$ is invertible and
\begin{align*}
 \lefteqn{(z + \varphi(A))^{-1}} \\
 &= (z+\varphi(\infty))^{-1}I  +\frac{q}{\pi}\int_0^\infty 
\frac{\operatorname{Im} \varphi(te^{i\pi/q})\,t^{q-1}}
{(z+\varphi(te^{i\pi/q}))(z+\varphi(te^{-i\pi/q}))} \,(A^q+t^q)^{-1}\,dt \nonumber\\
 &= (z+\varphi(0+))^{-1}I  -\frac{q}{\pi}\int_0^\infty
\frac{\operatorname{Im} \varphi(te^{i\pi/q})}
{(z+\varphi(te^{i\pi/q}))(z+\varphi(te^{-i\pi/q}))t} \,A^q (A^q+t^q)^{-1}\,dt.  \nonumber
\end{align*}
Here the integrals are absolutely convergent.
\end{thm}

\begin{cor}\label{TCbf1}
Let $A\in {\rm Sect}(\omega)$, $\omega\in [0,\pi)$, 
and let $\varphi\in \mathcal{CBF}$. Then $\varphi(A)\in {\rm Sect}(\omega)$.

More precisely, when $\theta \in (0,\pi-\omega)$ and 
\[
\max \left( \frac{\pi}{\pi-\theta}, 1 \right) < q < \frac{\pi}{\omega},
\]
one has
\begin{align*}
\label{Ra110}
\left\|z(z+\varphi(A))^{-1}\right\|  
\le \frac{1}{\sin(\pi/q)} +\frac{2q \tan(\pi/(2q)) \tilde{M}_{\omega,q,0}(A)} {\pi \cos(\pi/q)\cos^2\left((\pi/q+\theta)/2\right)},
\quad  z\in \Sigma_\theta,
\end{align*}
where $\tilde{M}_{\omega,q,0}(A)$ is defined by {\rm(\ref{Mn})} in the Appendix.
\end{cor}

\begin{example}
As in Example \ref{exlog}, consider the complete Bernstein function $\varphi(z) = \log(1+z)$, and let $q=2$.  For $A \in \operatorname{Sect}(\omega)$, where $\omega \in (0,\pi/2)$, and $z\in\C_+$, we have
\[
(z + \log(1+A))^{-1} = \int_0^{\pi/2} \frac{s \sin s} {\left( (z - \log \cos s)^2 + s^2\right) (\cos s)^3 } (A^2 + \tan^2 s)^{-1} \, ds.
\]
This may be compared with \cite[Example 2, p.501]{Mi98}.
\end{example}

Finally we give an application which is a continuous-time counterpart of \cite[Theorem 1.2]{GT2}.  An operator $T \in \mathcal{L}(X)$ is said to be a \emph{Ritt operator of angle} $\theta \in (0,\pi/2]$  if $\sigma(T) \subset \{\lambda \in 1- \overline{\Sigma}_\theta: |\lambda| \le 1\}$ and for each $\theta' \in (\theta,\pi/2)$ there exists $C_{\theta'}$ such that 
\[
\|(\lambda-T)^{-1} \| \le \frac{C_{\theta'}}{|\lambda-1|}, \qquad \lambda \in 1- \Sigma_{\theta'}, \;|\lambda|\le1.
\]
The class of Ritt operators has received considerable attention in recent years
due to its similarities with the class of generators of sectorially bounded
holomorphic semigroups. See \cite{GT2} for more on its properties and relevance.

\begin{thm}  \label{TCbf14}
Let $(S(t))_{t\ge0}$ be a sectorially bounded holomorphic semigroup of angle $\theta \in (0,\pi/2]$, and $\mu$ be a probability measure on $[0,\infty)$.   Let  
\[
Tx = \int_0^\infty  S(t)x \, \mu(dt), \qquad  x \in X.
\]
Then  $T$ is a Ritt operator of angle $\frac{\pi}{2}-\theta$.

\begin{proof}
Let $f \in \mathcal{BF}$ have the L\'evy triple $(0,0,\mu)$.  Then, by (\ref{bpfc}),
\[
Tx = x - \int_0^\infty (1 - S(t))x \,\mu(dt) = x - f(A)x,
\]
for $x \in \dom(A)$ and then for $x \in X$ by density.  By Corollary \ref{TCbf12}, $I - T = f(A) \in \operatorname{Sect}(\theta)$.  By the spectral mapping theorem \cite[Theorem 16.4.1]{Hille},
\[
\sigma(T) \subset \{\lambda \in \C: |\lambda|<1\} \cup \{1\}.
\]
  Then $T$ is Ritt of angle $\frac{\pi}{2}-\theta$ \cite[Theorem 4.1]{GT2}.
\end{proof}
\end{thm}

\section{Appendix: Fractional Powers} \label{frac}

In this section we present quantitative versions of Proposition \ref{PrBBL},
with explicit constants. We were not able to find such explicit estimates
in the literature.  In this paper we use only the case considered in Proposition \ref{PrBBL10} where $q>1$, but for completeness we give first the case where $q <1$, which was first considered by Kato \cite{Kato1960}.

\begin{prop}\label{PrBBL1}
Let $A$ be sectorial, so that
\begin{equation*}
M(A)=\sup_{s>0}\|s(A+s)^{-1}\|<\infty.
\end{equation*}
If $q\in (0,1)$ then $A^q \in {\rm Sect}(q\pi)$, and for every $\psi \in (0,(1-q)\pi)$ one has
\begin{equation}\label{Zzz}
\|(A^q+z)^{-1}\|\le \frac{\sin(\pi q)}{\pi q} \frac{(\pi
q+\psi)}{\sin(\pi q+\psi)}\frac{M(A)}{|z|}, \qquad z\in
\Sigma_\psi.
\end{equation}
\end{prop}

\begin{proof}
The property $A^q \in {\rm Sect}(q\pi)$ is part of Proposition \ref{PrBBL},  so we need only to  prove the resolvent estimate \eqref{Zzz}. If $z\in \Sigma_{\psi}$, then by \cite[Theorem 2]{Kato1960} we have
\begin{eqnarray}\label{Kr10}
(A^q+z)^{-1}&=& \frac{\sin(\pi q)}{\pi} \int_0^\infty\frac{t^q
(A+t)^{-1}\,dt}
{(t^q e^{i\pi q}+z)(t^q e^{-i\pi q}+z)}.
\end{eqnarray}
Since
\[
|te^{i\pi q} + z|    = \left| te^{i(\pi q  - \arg z)} +|z|\right|     \ge  \left| te^{i(\pi q  + \psi)} +|z|\right|, \qquad z\in\Sigma_\psi, \;t>0,
\]
and     
\[
\int_0^\infty \frac{dt}{|te^{i\theta}+s|^2}=\frac{\theta}{s\sin \theta},
\quad \theta\in (-\pi,\pi),\quad s>0
\]
(see (\cite[Formula 2.2.9.25]{Prud}), (\ref{Kr10})
implies that
\begin{align*}
\|(A^q +z)^{-1}\|&\le \frac{M\sin(\pi q)}{\pi}
\int_0^\infty\frac{t^{q-1}\,dt} {|t^q e^{i\pi q}+z| |t^q e^{-i\pi
q}+z|}
\\
&\le \frac{M\sin(\pi q)}{\pi} \int_0^\infty\frac{t^{q-1}\,dt}
{|t^q e^{i(\pi q+\psi)}+|z||^2}\\
&= \frac{M\sin(\pi q)}{\pi q} \int_0^\infty\frac{dt}
{|t e^{i(\pi q+\psi)}+|z||^2}\\
&=\frac{M\sin(\pi q)}{\pi q} \frac{\pi q +\psi}{\sin(\pi
q +\psi)}\frac{1}{|z|}, \qquad z\in \Sigma_\psi.  \qedhere
\end{align*}
\end{proof}

The next result is a quantitative version of Proposition
\ref{PrBBL} for $q>1$.

\begin{prop}\label{PrBBL10}
Let $A\in \Sect(\omega)$, $\omega\in [0,\pi)$. so that
\begin{equation}\label{EsR0}
\|z(A+z)^{-1}\|\le M(A,\omega'),\qquad z\in
\Sigma_{\pi-\omega'},\quad \omega'\in (\omega,\pi).
\end{equation}
Let $q>1$ be such that $q\omega<\pi$. Then $A^q\in \Sect(q\omega)$
and, moreover, for every $\psi \in (0,\pi-q\omega)$,
\begin{equation}\label{MEs0}
\|z(A^q+z)^{-1}\|\le \tilde{M}_{\omega,q,\psi}(A),\qquad z \in \Sigma_{\psi},
\end{equation}
where
\begin{align}
\label{Mn}
\tilde{M}_{\omega,q,\psi}(A) &:=
M(A)+\frac{2M(A,\beta_{\omega,q,\psi})}{\pi\cos(\beta_{\omega,q,\psi}/2)
\cos (q\beta_{\omega,q,\psi}/2)}, \\
\beta_{\omega,q,\psi}&:=\frac{\omega+(\pi-\psi)/q}{2}.  \nonumber
\end{align}
\end{prop}

\begin{proof}
Proposition \ref{PrBBL} gives that $A^q\in \Sect(q\omega)$, and it remains to prove only
\eqref{MEs0}. Using the holomorphic functional calculus and \cite[Proposition 3.1.2]{Ha06},
we have
\[
z (A^q+z)^{-1}=|z|^{1/q}(A+|z|^{1/q})^{-1}+f_{z,q}(A),\qquad z\in
\Sigma_{\pi-q\omega},
\]
where
\[
f_{z,q}(\lambda):=\frac{z}{z+\lambda^q}-\frac{|z|^{1/q}}{\lambda+|z|^{1/q}}=
\frac{z\lambda -z^{1/q}\lambda^q}{(z+\lambda^q)(\lambda+|z|^{1/q})}.
\]
Hence,
\[
\|z (A^q+z)^{-1}\|\le M(A)+\|f_{z,q}(A)\|,\quad
z \in\Sigma_{\pi-q\omega}.
\]

Let $\psi \in (0,\pi-q\omega)$ be fixed, and let $\theta\in (\omega,(\pi-\psi)/q)$.  For $z \in \Sigma_\psi$,  we have
\begin{align*}
\|f_{z,q}(A)\|&\le \frac{1}{2\pi}\int_{\partial \Sigma_{\theta}}
|f_{z,q}(\lambda)|\,\|(\lambda-A)^{-1}\|\,|d\lambda|\\
&\le \frac{M(A,\theta)}{2\pi}\int_{\partial \Sigma_{\theta}}
|f_{z,q}(\lambda)|\frac{|d\lambda|}{|\lambda|}.
\end{align*}
Using (\ref{ts1}), and setting
$c=\cos(\theta/2)\cos((q\theta+\psi)/2)$ and $\lambda = t|z|^{1/q} e^{\pm i\theta}$, we have
\begin{align*}
\int_{\partial \Sigma_{\theta}} |f_{z,q}(\lambda)|\frac{|d\lambda|}{|\lambda|} &\le
\int_{\partial \Sigma_{\theta}}
\frac{|z|+|z|^{1/q}|\lambda|^{q-1}}{|z+\lambda^q|
|\lambda+|z|^{1/q}|} |d\lambda|\\
&\le \frac{1}{c}\int_{\partial \Sigma_{\theta}}
\frac{|z|+|z|^{1/q}|\lambda|^{q-1}}{(|z|+|\lambda|^q|)
(|\lambda|+|z|^{1/q})} |d\lambda|
\\
&=\frac{2}{c}\int_0^\infty\frac{1+t^{q-1}}{(t^q+1) (t+1)}\,dt\\
&\le \frac{2}{c} \int_0^\infty \frac{2}{(t+1)^2} \, dt \\
&= \frac{4}{c} \,.
\end{align*}
Thus,
\[
\|z (A^q+z)^{-1}\|\le M(A)+\frac{2
M(A,\theta)}{\pi\cos(\theta/2)\cos ((q\theta+\psi)/2)}, \qquad
z \in\Sigma_\psi.
\]
Choosing  $\theta=\beta_{\omega,q,\psi}$. we obtain (\ref{MEs0}) and  (\ref{Mn}).
\end{proof}

\end{document}